\documentclass{amsart}

\usepackage{amsmath, amsthm, amssymb}
\usepackage{amsfonts}
\usepackage{amsmath}
\usepackage{amssymb}
\usepackage{array}
\usepackage{stmaryrd}
\usepackage{graphicx}
\usepackage{color}
\usepackage{tikz}

\usepackage{enumitem}

\newtheorem{theorem}{Theorem}[section]
\newtheorem{lemma}[theorem]{Lemma}

\newtheorem{corollary}[theorem]{Corollary}

\newtheorem{remark}[theorem]{\it Remark}
\newtheorem{theorema}{Theorem}

\numberwithin{equation}{section}

\newcommand{\pt}{\partial}

\newcommand {\beq} {\begin{equation}}
\newcommand {\eeq} {\end{equation}}
\renewcommand{\sim}{\simeq}

\newcommand {\A} {{\mathcal B}}
\newcommand {\B} {{\mathcal A}}
\newcommand {\F} {{\mathcal F}}
\newcommand{\ttau}{{\tilde\tau}}
\newcommand {\U} {{\mathcal U}}

\newcommand{\LL}{{\mathcal L}}

\newcommand{\RR}{{\mathcal R}}
\newcommand{\R}{\mathbb{R}}

\definecolor{blue}{rgb}{0,0,0.7}

\begin{document}

\title[L2 method for a fractional-derivative problem]%
{
Error analysis of an L2-type method on graded meshes for a fractional-order parabolic problem}

\author{Natalia Kopteva}
\address{Department of Mathematics and Statistics,
University of Limerick, Limerick, Ireland}
\email{natalia.kopteva@ul.ie}

\thanks{This research was  supported by
Science Foundation Ireland Grant SFI/12/IA/1683.}


\date{}



\maketitle

\begin{abstract}
An initial-boundary value problem with a Caputo time derivative of fractional order $\alpha\in(0,1)$ is considered,
 solutions of which typically exhibit a singular behaviour at an initial time.
 An L2-type discrete fractional-derivative operator of order $3-\alpha$ is
 considered on nonuniform temporal meshes.
 Sufficient conditions for the inverse-monotonicity of this operator are established,
 which yields sharp pointwise-in-time error bounds
on quasi-graded temporal meshes with arbitrary degree of grading.
In particular, those results imply that milder (compared to the optimal) grading yields optimal convergence rates in positive time.
Semi-discretizations in time and full discretizations are addressed.
The theoretical findings are illustrated by numerical experiments.
\end{abstract}

\section{Introduction}

The Caputo time derivative of fractional order $\alpha\in(0,1)$, which will be denoted by $D_t^\alpha$, is  defined \cite{Diet10} by
\begin{equation}\label{CaputoEquiv}
D_t^{\alpha} u(\cdot,t) :=  \frac1{\Gamma(1-\alpha)} \int_{0}^t(t-s)^{-\alpha}\, \pt_s u(\cdot, s)\, ds
    \qquad\text{for }\ 0<t \le T,
\end{equation}
where $\Gamma(\cdot)$ is the Gamma function, and $\pt_s$ denotes the partial derivative in $s$.

The paper is devoted to the analysis of an L2-type
discrete fractional-derivative operator for $D_t^{\alpha}$ from \cite{higher_order},
 based on piecewise-quadratic Lagrange interpolants.
 In \cite{higher_order}, this operator
is analysed  on uniform temporal meshes, and the optimal convergence order
 $3-\alpha$ in time is established
under strong regularity assumptions on the exact solution. (Similar L2-type discretizations of order $3-\alpha$
on uniform temporal meshes
were considered, e.g., in articles \cite{jcp_gao,jcp_xing_yan}, the latter  giving optimal error bounds in positive time taking into account
more realistic low regularity of the exact solution.)

The purpose of this paper is consider this discrete fractional-derivative operator on more general quasi-graded temporal meshes.
For this, we employ
the framework from the recent paper \cite{NK_XM} (which builds on the analysis of \cite{NK_MC_L1}, and, to some degree,~\cite{ChenMS_JSC}).
This approach is based on barrier functions for derivation of subtle stability properties, and allows, in a relatively simple way, to
get sharp pointwise-in-time error bounds on quasi-graded temporal meshes with arbitrary degree of grading.
\begin{itemize}[leftmargin=0.7cm]
\item
However, compared to the two methods considered in \cite{NK_XM}, the L1 scheme and the Alikhanov L2-1${}_\sigma$ scheme, now we have
a significantly more challenging case, as the considered discrete fractional-derivative operator 
is not associated with an M-matrix. So our main challenge in this paper will be to establish
the inverse-monotonicity of the discrete operator on nonuniform meshes.

\item
For the same reason, the generalization of our error analysis to the parabolic case also becomes substantially more challenging.
\end{itemize}
\noindent
Note that the inverse-monotonicity on uniform temporal meshes was established in \cite{higher_order}.
However, the evaluations in the latter article are quite intricate, 
so it is not clear whether they can be generalized to  more general meshes.
We take a very different route and
employ a non-standard set of basis functions (see Fig.\,\ref{fig_basis}),
which very naturally leads to a representation of  the discrete operator as a product of two M-matrices.
{\color{blue}To be more precise, the discrete version $\delta_t^\alpha$ of the Caputo fractional-derivative operator $D_t^\alpha$
will be represented in the form
\begin{subequations}\label{UV_ab}
\beq\label{UV}
\delta_t^\alpha U^m=\sum_{j=0}^m \kappa_{m,j}V^j\;\;\forall\,m\ge 1,
\quad V^j:=\frac{U^j-\beta_j U^{j-1}}{1-\beta_j}\;\;\forall\,j\ge 1,\quad
V^0:=U^0,
\eeq
where $\beta_j\in [0,1)$.
Then relatively simple sufficient conditions
will be formulated
 for
choosing a set $\{\beta_j\}$ such that
\beq\label{UV_kappa}
\kappa_{m,m}>0\;\;\mbox{and}\;\;\sum_{j=0}^{m}\kappa_{m,j}=0\;\;\forall\,m\ge 1,\quad \kappa_{m,j}\le 0\;\;\forall\,0\le j<m\le M.
\eeq
\end{subequations}%
As the representation \eqref{UV_ab} immediately implies that $\delta_t^\alpha$ is associated with an inverse-monotone matrix (see Remark~\ref{rem_M_matrices}),
the required stability properties of the discrete fractional-derivative operator follow, which  enables us to employ the error analysis framework from~\cite{NK_XM}.}%



This error analysis will be applied for 
the fractional-order parabolic problem
\beq\label{problem}
\begin{array}{l}
D_t^{\alpha}u+\LL u=f(x,t)\quad\mbox{for}\;\;(x,t)\in\Omega\times(0,T],\\[0.2cm]
u(x,t)=0\quad\mbox{for}\;\;(x,t)\in\pt\Omega\times(0,T],\qquad
u(x,0)=u_0(x)\quad\mbox{for}\;\;x\in\Omega.
\end{array}
\eeq
This problem is posed in a bounded Lipschitz domain  $\Omega\subset\R^d$ (where $d\in\{1,2,3\}$).
The spatial operator $\LL$ here is a linear second-order elliptic operator defined by
\beq\label{LL_def}
\LL u := \sum_{k=1}^d \bigl\{-\pt_{x_k}\!(a_k(x)\,\pt_{x_k}\!u)  \bigr\}+c(x)\, u,
\eeq
with sufficiently smooth coefficients $\{a_k\}$ and $c$ in $C(\bar\Omega)$, for which we assume that $a_k>0$
and  $c\ge 0$ in $\bar\Omega$.

The L2-type fractional-derivative operator that we consider, denoted $\delta_{t}^{\alpha}$, is defined as follows.
On the temporal mesh $0=t_0<t_1<\ldots <t_M=T$,  $\forall\,m= 1,\ldots, M$ let
\begin{subequations}\label{delta_t_def}
\beq
\delta_{t}^{\alpha} U^m :=D^\alpha_t (\Pi^mU)(t_m),\quad
\quad
\Pi^m:=\left\{\begin{array}{cll}
\Pi_{1,1}&\mbox{on~}(0,t_1)&\mbox{for~}m=1,\\
\Pi_{2,j}&\mbox{on~}(t_{j-1},t_j)&\mbox{for~}1\le j<m,\\
\Pi_{2,j-1}&\mbox{on~}(t_{j-1},t_j)\;&\mbox{for~}j=m>1,\\
\end{array}\right.
\eeq
where
 $\Pi_{1,j}$ and $\Pi_{2,j}$ are the standard linear and quadratic Lagrange interpolation operators with the following
 interpolation points:
\beq
\Pi_{1,j}\;:\;\{t_{j-1},t_j\},\qquad\qquad
\Pi_{2,j}\;:\;\{t_{j-1},t_j,t_{j+1}\}.
\eeq
\end{subequations}

Similarly to \cite{stynes_etal_sinum17,NK_MC_L1,ChenMS_JSC}, our main interest will be in graded temporal meshes as they offer an efficient way of computing reliable numerical approximations of solutions singular at $t=0$,
which is typical for \eqref{problem}.
It should be noted that these three papers are concerned with global-in-time error bounds on graded meshes. There is also a lot of interest in the literature  in optimal error bounds in positive time on uniform meshes;
see, e.g. \cite{gracia_etal_cmame,laz_review,NK_MC_L1}.
By contrast,
here, following the recent paper \cite{NK_XM},  pointwise-in-time error bounds will be obtained, while
an arbitrary degree of mesh grading (with uniform meshes included as a particular case) is allowed.
In particular, our results imply that milder (compared to the optimal) grading yields optimal convergence rates in positive time; see Remarks~\ref{rem_positive_time} and~\ref{rem_global_time}.

Throughout the paper, it is assumed that there exists a unique solution of this problem 
such that $\|\pt_t^l u(\cdot,t)\|_{L_2(\Omega)}\lesssim 1+t^{\alpha-l}$ for $l\le 3$.
This is a realistic assumption, satisfied by typical solutions of problem \eqref{problem},
in contrast to stronger assumptions of type $\|\pt^l u(\cdot,t)\|_{L_2(\Omega)}\lesssim 1$ frequently made in the literature
(see, e.g., references in \cite[Table~1.1]{laz_2fully_16}).
Indeed,
 \cite[Theorem~2.1]{stynes_too_much_reg} shows
that
if a solution $u$ of \eqref{problem} is less singular than we assume, 
then the initial condition $u_0$ is uniquely defined by the other data of the problem, which is clearly too restrictive.
At the same time, our results can be easily applied to the case of $u$ having no
singularities or exhibiting a somewhat different singular behaviour at $t=0$
{\color{blue}(see Remark~\ref{rem_gen_sing}).}
\smallskip

{\it Outline.}
Sufficient conditions for inverse-monotonicity of the discrete fractional-derivative operator are established in \S\ref{sec_inv_monot},
which enables us to establish its stability properties on quasi-graded meshes in~\S\ref{sec_stab}.
Error analysis for a simplest example without spatial derivatives is given  in \S\ref{sec_paradigm},
while semi-discretizations in time and full discretizations for the parabolic case are addressed in \S\ref{sec_parabolic}.
Finally, our theoretical findings are illustrated by numerical experiments in \S\ref{sec_Num}.
\smallskip

{\it Notation.}
We write
 $a\sim b$ when $a \lesssim b$ and $a \gtrsim b$, and
$a \lesssim b$ when $a \le Cb$ with a generic positive constant $C$ depending on $\Omega$, $T$, $u_0$ and
$f$,
but not 
%
 on the total numbers of degrees of freedom in space or time.
  Also, for 
  $k \ge 0$,
  we shall use the standard norms
  in the space $L_2(\Omega)$ and the related Sobolev spaces $W_2^k(\Omega)$,
  while  $H^1_0(\Omega)$ is the standard space of functions in $W_2^1(\Omega)$ vanishing on $\pt\Omega$.

\section{Inverse-monotonicity of the discrete fractional-derivative operator}\label{sec_inv_monot}
In this section we shall establish sufficient conditions on the temporal mesh $\{t_j\}_{j=0}^M$
for the inverse-monotonicity of the discrete fractional-derivative operator $\delta_t^\alpha$. The latter is understood
in the sense that the matrix associated with $\delta_t^\alpha$ is inverse-monotone, i.e. all elements of the inverse of
this matrix are non-negative.

The following notation for the temporal mesh will be used throughout the paper:
\beq\label{grid_notation}
\tau_j:=t_j-t_{j-1},\quad \ttau_j:={\textstyle\frac12}(\tau_{j-1}+\tau_j),\quad
\rho_j:=\frac{\tau_{j}}{\tau_{j-1}},
\quad
\sigma_j:=\frac{\tau_{j}-\tau_{j-1}}{\tau_j+\tau_{j-1}}
=1-\frac{2}{1+\rho_j}.
\eeq

\subsection{Matrix product representation for the discrete fractional-deri\-vat\-ive operator}
{\color{blue}Our first task  will be to
find a representation for $\delta_{t}^{\alpha}$ in the form \eqref{UV}, where the
set of real numbers $\{\beta_j\}_{j=0}^M$, with $\beta_j\in[0,1)$ and $\beta_0=0$,
is such that \eqref{UV_kappa} is satisfied.}

%
\begin{remark}[Inverse monotonicity]\label{rem_M_matrices}
Set $F^m:=\delta_t^\alpha U^m$ for $m=1,\ldots,M$ and
augment these equations by $F^0=U^0$. Now \eqref{UV_ab} yields
 the representation
$\vec{F}=A_1\vec{V}$ with $\vec{V}=A_2\vec{U}$, or simply $\vec{F}=A_1 A_2\vec{U}$,
where $A_1$ and $A_2$ are
 $(M+1)\times (M+1)$ matrices, and the notation of type
 $\vec{U}:=\{U^j\}_{j=0}^M$ is used for the corresponding column vectors.
Being M-matrices (i.e. diagonally dominant, with non-positive off-diagonal elements),
both $A_1$ and $A_2$ are inverse-monotone, hence the product $A_1A_2$
is also inverse-monotone (i.e. the elements of its inverse are non-negative).
Thus \eqref{UV_ab} implies that the operator $\delta_t^\alpha$
is associated with an inverse-monotone matrix.
\end{remark}

To describe  a representation of type \eqref{UV} in a simple way on an arbitrary temporary mesh, we shall
employ a \underline{non-standard basis} $\{\Phi^j(t_k)\}_{j=0}^M$
for  functions in $\R^{M+1}$ associated with the mesh $\{t_k\}_{k=0}^M$, which is defined by
\beq\label{Phi_basis}
\Phi^j(t_k):=0\;\mbox{for}\;k\le j-1,\quad \Phi^j(t_j):=1,\quad
\Phi^j(t_k):=\beta_k\Phi^j(t_{k-1})\;\mbox{for}\;k\ge j+1
\eeq
(see Fig.\,\ref{fig_basis} (left)).

  \begin{figure}[t!]
\noindent%
\begin{tikzpicture}[scale=0.23]
  \draw[->] (-0.2,0) -- (24.5,0) ;
  \draw[->] (0,-0.2) -- (0,11.5) ;
  \node[left] at (0,10) {$1$};
  \node[left] at (0,6) {$\beta_{j+1}$};
  \node[left] at (0,3.4) {$\beta_{j+1}\beta_{j+2}$};
  \node[left] at (0,1.2) {$\beta_{j+1}\beta_{j+2}\beta_{j+3}$};
\path[draw, thick] (0,0)--(5.5,0)--(8.5,10)--(12,6)--(16,3)--(21,1.5)--(24,1);

\draw[fill] (0,0) circle [radius=0.3] node[below] {$0$};
\draw[fill] (1.5,0) circle [radius=0.3];
\draw[fill] (3.5,0) circle [radius=0.3];
\draw[fill] (5.5,0) circle [radius=0.3] node[below] {$t_{j-1}$};
\draw[fill] (8.5,10) circle [radius=0.3]node[above right] {$\Phi^j$};
\draw[fill] (16,3) circle [radius=0.3];
\draw[fill] (12,6) circle [radius=0.3];
\draw[fill] (21,1.5) circle [radius=0.3];

\path[draw,help lines]  (-0.3,10)--(8.9,10);
\path[draw,help lines]  (-0.3,6)--(12.4,6);
\path[draw,help lines]  (-0.3,3)--(16.4,3);
\path[draw,help lines]  (-0.3,1.5)--(21.4,1.5);

\path[draw,help lines]  (8.5,-0.3)--(8.5,10.4);
\path[draw,help lines]  (12,-0.3)--(12,6.4);
\path[draw,help lines]  (16,-0.3)--(16,3.4);
\path[draw,help lines]  (21,-0.3)--(21,1.9);

\node[below]  at (8.5,0) {$t_{j}$};
\node[below]  at (12,0) {$t_{j+1}$};
\node[below]  at (16,0) {$t_{j+2}$};
\node[below]  at (21,0) {$t_{j+3}$};
\end{tikzpicture}%
\hfill%
\begin{tikzpicture}[scale=0.23]
  \draw[->] (-0.2,0) -- (18,0) ;
  \draw[->] (0,-0.2) -- (0,11.5) ;
  \node[left] at (0,10) {$1$};
\path[draw, thick] (0,0)--(5.5,0)--(8.5,10)--(12,0)--(16,0)--(17,0);

\draw[fill] (0,0) circle [radius=0.3] node[below] {$0$};
\draw[fill] (1.5,0) circle [radius=0.3];
\draw[fill] (3.5,0) circle [radius=0.3];
\draw[fill] (5.5,0) circle [radius=0.3] node[below] {$t_{j-1}$};
\draw[fill] (8.5,10) circle [radius=0.3]node[above right] {$\phi^j$};
\draw[fill] (16,0) circle [radius=0.3];
\draw[fill] (12,0) circle [radius=0.3];

\path[draw,help lines]  (-0.3,10)--(8.9,10);
\path[draw,help lines]  (8.5,-0.3)--(8.5,10.4);

\node[below]  at (8.5,0) {$t_{j}$};
\node[below]  at (12,0) {$t_{j+1}$};
\node[below]  at (16,0) {$t_{j+2}$};
%
\end{tikzpicture}%
\vspace{-0.3cm}
 \caption{\label{fig_basis}\it\small
Non-standard basis $\{\Phi^j\}$  from \eqref{Phi_basis} (left) and
 hat-function basis $\{\phi^j\}$ (right).}
 \end{figure}
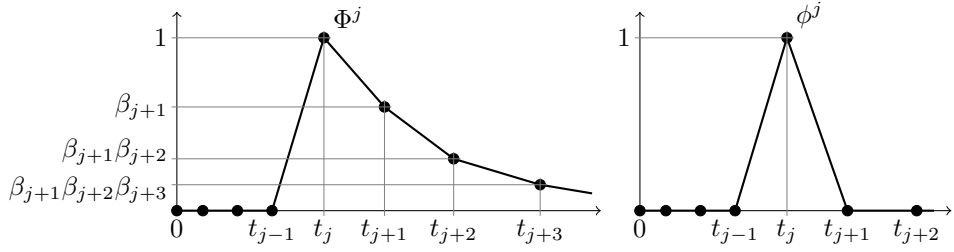

\begin{lemma}
Given a set  $\{\beta_j\}_{j=1}^M$ with $\beta_j\in[0,1)$  and the basis \eqref{Phi_basis}, the coefficients $\kappa_{m,j}$ in \eqref{UV}
are described by
\beq\label{kappa_basis}
\frac{\kappa_{m,j}}{1-\beta_j}=D_t^\alpha(\Pi^m\Phi^j)(t_m)\qquad\forall\,0\le j\le m\le M.
\eeq
\end{lemma}
\begin{proof}
The definition of $\{V^j\}$ in \eqref{UV} is equivalent to the following basis expansion of $\{U^j\}$:
\beq\label{UV2}
U^k=\sum_{j=0}^M V^j(1-\beta_j)\,\Phi^j(t_k)\qquad \forall\,k=0,\ldots M.
\eeq
Indeed, by \eqref{Phi_basis}, for $k=0$ this yields $U^0=V^0(1-\beta_0)=V^0$, while for $k\ge 1$, in view of $\Phi^j(t_k)=0$ for $j>k$, one
can replace $\sum_{j=0}^M$ in \eqref{UV2} by $\sum_{j=0}^k$, so, indeed,
$$
U^k=\underbrace{\sum_{j=0}^{k-1} V^j(1-\beta_j)\underbrace{\Phi^j(t_k)}_{{}=\beta_k \Phi^j(t_{k-1})}}_{=\beta_k U^{k-1}}+V^k(1-\beta_k)
=\beta_kU^{k-1}+(1-\beta_k)V^k.
$$

Next,  \eqref{UV2} immediately implies that $\Pi^m U=\sum_{j=0}^M V^j(1-\beta_j)\Pi^m\Phi^j$ on $(0,t_m)$,
where $\Pi^m\Phi^j=0$  for $j>m$,
so
$$
\delta_t^m U^m
=D_t^\alpha (\Pi^m U)(t_m)
= \sum_{j=0}^m V^j(1-\beta_j)\,D_t^\alpha(\Pi^m\Phi^j)(t_m),
$$
which, compared with \eqref{UV}, immediately yields \eqref{kappa_basis}.
\end{proof}

It will be convenient to formulate sufficient conditions for \eqref{UV_kappa} in terms of the \underline{standard hat-function basis}
 $\{\phi^j(t_k)\}_{j=0}^M$
for  functions in $\R^{M+1}$ associated with the mesh $\{t_k\}_{k=0}^M$,
 i.e. $\phi^j(t_k)$ equals 1 if $k=j$ and $0$ otherwise
(see Fig.\,\ref{fig_basis} (right)).

\begin{lemma}\label{lem_sufficient}
Let the temporal mesh satisfy $\rho_j\ge\rho_{j+1}\ge1$ $\forall\,j\ge2$.
Then representation \eqref{UV} satisfies \eqref{UV_kappa} if
\begin{subequations}\label{key}
\begin{align}\label{key1}
\delta_t^\alpha \phi^{m-1}(t_m)+\beta_{m}\delta_t^\alpha\phi^{m}(t_m)&< 0\qquad\mbox{for~~}m\ge1,\\\label{key2}
\delta_t^\alpha \phi^{m-2}(t_m)+\beta_{m-1}\Bigl[\delta_t^\alpha \phi^{m-1}(t_m)+\beta_{m}\delta_t^\alpha\phi^{m}(t_m)\Bigr]&\le 0\qquad\mbox{for~~}m\ge2,
\end{align}
\end{subequations}
where  $\delta_t^\alpha \phi^k(t_m)$ $\forall\, k$ is understood as $D_t^\alpha (\Pi^m\phi^k)(t_m)$.
Under the above conditions we also have
\beq\label{kappa0}
-\kappa_{m,0}\gtrsim t_m^{-\alpha}\qquad\quad\mbox{for~~}m\ge3.
\eeq
\end{lemma}
\begin{proof}
First, by \eqref{UV},
note that $V^j=1$ $\forall\,j$ implies that $U^j=1$ $\forall\,j$, which then, by \eqref{delta_t_def}, implies that $\delta_t^\alpha U^m=0$ $\forall\,m\ge1$,
so one gets $0=\sum_{j=0}^m\kappa_{m,j}\cdot 1$ $\forall\,m\ge1$, which immediately yields the second relation in \eqref{UV_kappa}.

Next, by \eqref{kappa_basis} combined with $1-\beta_j>0$ $\forall\,j$, we conclude that
$\kappa_{m,j}\le 0$  $\forall\, j<m$ is equivalent to $D_t^\alpha(\Pi^m\Phi^j)(t_m)\le 0$ $\forall\, j<m$.
To find sufficient conditions for the latter, note that \eqref{Phi_basis} implies that
\beq\label{Phi_phi}
\Phi^m(t_k)=\phi^m(t_k)\;\;\forall\,k\le m,\qquad
\Phi^j(t_k)=\phi^j(t_k)+\beta_{j+1}\Phi^{j+1}(t_k)\;\;\forall\,j,k\ge 0.
\eeq
In particular,  $\forall\,t_k\le t_m$ one has
$\Phi^{m-1}(t_k)=\phi^{m-1}(t_k)+\beta_{m}\phi^{m}(t_k)$
and
$\Phi^{m-2}(t_k)=\phi^{m-2}(t_k)+\beta_{m-1}\Phi^{m-1}(t_k)$,
so conditions \eqref{key1} and \eqref{key2}
are respectively equivalent to $D_t^\alpha(\Pi^m\Phi^{m-1})(t_m)< 0$ and $D_t^\alpha(\Pi^m\Phi^{m-2})(t_m)\le 0$.
Once the latter two inequalities hold true, an argument by induction shows that
for
$D_t^\alpha(\Pi^m\Phi^{j})(t_m)\le 0$ $\forall\, j\le m-3$
it suffices to check that $\delta_t^\alpha \phi^{j}(t_m)\le 0$ $\forall\, j\le m-3$.
The latter is true under the condition $\rho_j\ge\rho_{j+1}\ge1$ $\forall\,j\ge2$,
by \cite[Lemma~4]{ChenMS_JSC} (see also Remark~\ref{rem_JSC}).

To complete the proof of \eqref{UV_kappa}, note that one can replace $\kappa_{m,m}>0$ in \eqref{UV_kappa} by $\kappa_{m,m-1}<0$, the latter being satisfied
due to the strict inequality in \eqref{key1}.

For \eqref{kappa0}, let  $m\ge 3$
and
note that the above argument, in particular
the second relation in \eqref{Phi_phi} with $j=0$, implies
that $\delta_t^\alpha\Phi^{0}(t_m)\le\delta_t^\alpha\phi^{0}(t_m)\simeq -t_m^{-\alpha}$
(where we also used $\delta_t^\alpha\phi^{0}(t_m)=D_t^\alpha(\Pi^m\phi^{0})(t_m)\simeq -t_m^{-\alpha}$, which
can be shown on an arbitrary mesh from \eqref{delta_t_def}). Combining this bound with \eqref{kappa_basis} immediately yields
\eqref{kappa0}.
\end{proof}

\begin{remark}\label{rem_JSC}
In the statement of Lemma~\ref{lem_sufficient}, the assumption that $\rho_j\ge\rho_{j+1}\ge1$ $\forall\,j\ge2$
is only required for
$\delta_t^\alpha \phi^j(t_m)\le 0\;\;\forall j\le m-3$.
For the latter we use
\cite[Lemma~4]{ChenMS_JSC}, which is obtained for the Alikhanov scheme, but we rely on the fact that if, using the notation of \cite{ChenMS_JSC},
$\sigma=1$, then the coefficients $\kappa^*_{m,j}$ in the representation of type $\delta_t^\alpha U^m=\sum_{j=0}^m\kappa^*_{m,j}U^j$
are the same for the Alikhanov scheme and our scheme $\forall\,j\le m-3$, and, furthermore, $\kappa^*_{m,j}=\delta_t^\alpha \phi^j(t_m)$.
Note also that the above assumption on $\{\rho_j\}$
may be replaced by a weaker assumption; see \cite[(12), (16) and Remark~3]{ChenMS_JSC}.
\end{remark}

It is convenient to rewrite conditions \eqref{key} using the notation
\begin{subequations}\label{AB_def}
\begin{align}\label{AB_def_a}
\B_m&:= \ttau_{m}^{\alpha}\,\Gamma(1-\alpha)\,
2^{\alpha}\, \delta_t^\alpha \phi^m(t_m),\\
-\A_m&:= \ttau_{m}^{\alpha}\,\Gamma(1-\alpha)\,
2^{\alpha}\, \delta_t^\alpha \phi^{m-1}(t_m)&\hspace{-2cm}\mbox{for~~}m\ge1,\\
\F_m&:= \ttau_{m}^{\alpha}\,\Gamma(1-\alpha)\,
2^{\alpha}\, \delta_t^\alpha [\phi^{m-2}+\phi^{m-1}+\phi^{m}](t_m)
&\hspace{-0.5cm}\mbox{for~~}m\ge2,
\end{align}
\end{subequations}
where $\ttau_1:=\tau_1$ and $\ttau_m={\textstyle\frac12}(\tau_{m-1}+\tau_m)$ for $m\ge 2$ is from \eqref{grid_notation}.

\begin{corollary}\label{cor_suf_cond}
Let the temporal mesh satisfy $\rho_j\ge\rho_{j+1}\ge1$ $\forall\,j\ge2$.
Then 
representation \eqref{UV} satisfies \eqref{UV_kappa} if
\begin{subequations}\label{key_AB}
\begin{align}\label{key_AB_1}
\A_m-\beta_m\B_m&> 0\qquad\mbox{for~~}m\ge1,\\\label{key_AB_2}
(\A_m-\B_m+\F_m)-\beta_{m-1}(\A_m-\beta_m\B_m)&\le 0\qquad\mbox{for~~}m\ge2.
\end{align}
\end{subequations}
\end{corollary}

\begin{remark}
Combining \eqref{kappa_basis} with \eqref{Phi_phi} and \eqref{AB_def}, from the proof of Lemma~\ref{lem_sufficient}
one  gets
\beq\label{kappa_obersve}
\Gamma(1-\alpha)\,
2^{\alpha}\,\frac{\kappa_{m,m}}{1-\beta_m}=\ttau_{m}^{-\alpha}\,\B_m,
\quad
\frac{\kappa_{m,m}}{1-\beta_m}\cdot\frac{1-\beta_{m-1}}{|\kappa_{m,m-1}|}
=\frac{\B_m}{|\A_m-\beta_m\B_m|}\,.
\eeq
\end{remark}

\subsection{Uniform temporal mesh}\label{ssec_inv_mon_uniform}
We shall first estimate the quantities in \eqref{AB_def} and check the inverse-monotonicity conditions \eqref{key_AB}
for the case of uniform temporal meshes.

\begin{lemma}[Uniform temporal mesh]\label{lem_uniform_AB}
Let $\tau_j=\tau=TM^{-1}$ $\forall j\ge1$.
Then for the quantities in \eqref{AB_def} one has
\beq\label{AB_bounds_uniform}
\B_1=\A_1>0;\quad
\B_m=\B,\;\;\; \A_m=\A'-\A''_m\ge \nu \A',\;\;\; \F_m\le 1+\A''_m\;\;\forall\,m\ge2,
\eeq
where 
\beq\label{AB_def_uniform}
\B:=\frac{\alpha+2}{(1-\alpha)(2-\alpha)},\quad \A':=\frac{4\alpha}{(1-\alpha)(2-\alpha)},\quad 0\le \A''_m\le \frac{\alpha}{24},
\quad
\textstyle
\nu:=1-\frac1{48}(1-\alpha).
\eeq
\end{lemma}

\begin{proof}
For $m=1$, we have $\Pi^m\phi^{0}(s)=1-s/t_1$ and $\Pi^m\phi^{1}(s)=s/t_1$ on $(0,t_1)$ (as here $\Pi^m=\Pi_{1,1}$), so
$\delta_t^\alpha\phi^{1}(t_1)=-\delta_t^\alpha\phi^{0}(t_1)>0$, so
$\B_1=\A_1>0$.

Now let $m\ge 2$
and combine \eqref{AB_def} with \eqref{delta_t_def} and \eqref{CaputoEquiv}.
Rewriting the resulting integrals in terms of a new variable $\hat s:=(s-t_{m-1})/\tau$, so  the interval $(t_{m-2},t_m)$ is mapped to $(-1,1)$,
while $\ttau_m^\alpha(t_m-s)^{-\alpha}=(1-\hat s)^{-\alpha}$,
a calculation shows that
\begin{subequations}\label{AB_calculation}
\begin{align}
\B_m&= 2^{\alpha}\! \int_{-1}^1\!\! (\hat s+{\textstyle\frac12})(1-\hat s)^{-\alpha}\,d\hat s =\B,\label{AB_calculation_B}\\
\A_m&= 2^{\alpha}\! \int_{-1}^1\!\! 2\hat s(1-\hat s)^{-\alpha}\,d\hat s-\A''_m =\A'-\A''_m.\label{AB_calculation_A}
\end{align}
Here we used the observations that
$\Pi^m \phi^m(\hat s)$ is $\frac12\hat s(\hat s+1)$ on $(-1,1)$ and vanishes otherwise,
while
 $\Pi^m \phi^{m-1}(\hat s)$ is $1-\hat s^2$ on $(-1,1)$
and  vanishes for $\hat s>1$.
For $m=2$ one has $\A''_2=0$, while $\A''_m$ for $m>2$
corresponds to $\Pi^m \phi^{m-1}(\hat s)=\frac12(\hat s+1)(\hat s+2)<0$
on $(-2,-1)$, so,
using integration by parts on this interval, we arrive at
\beq\label{AB_calculation_A2}
\A''_m
:=
-\alpha 2^{\alpha}\! \int_{-2}^{-1}\!\! \underbrace{\Pi^m \phi^{m-1}(\hat s)}_{{}<0}\, 
\underbrace{(1-\hat s)^{-\alpha-1}}_{{}< 2^{-\alpha-1}}d\hat s
\le -\alpha 2^{-1}\!\int_{-2}^{-1}\!\! \Pi^m \phi^{m-1}(\hat s)\, d\hat s\le \frac{\alpha}{24}\,,
\eeq
\end{subequations}
in view of 
$\int_{-2}^{-1}\Pi^m \phi^{m-1}(\hat s)\,d\hat s=-\frac1{12}$.
Note also that $\A_m''/\A'\le\frac1{96}(1-\alpha)(2-\alpha)\le 1-\nu$, so we get another desired assertion $\A'-\A''_m\ge \nu \A'$.

As to $\F_m$, set $\chi^{m-2}:=\phi^{m-2}+\phi^{m-1}+\phi^{m}$ and
note that $\chi^{m-2}(t_j)$ is 0 for $j<m-2$ and 1 for $j\ge m-2$.
So for $m=2$ one has $\chi^{m-2}=1$ on $(0,t_m)$ so $\F_m=0$.
Otherwise
$\frac{d}{d\hat s}\Pi^m\chi^{m-2}(\hat s)$ has support on $(-2,-1)$ for $m=3$ and on $(-3,-1)$ for $m>3$,
so we split
$\F_m=\F'_m+\F''_m$ with $\F''_3=0$ and
$$
\F_m':=
2^{\alpha}\! \int_{-2}^{-1}\!\! \underbrace{{\textstyle \frac{d}{d\hat s}}\Pi^m\chi^{m-2}(\hat s)}_{{>0}} \,\underbrace{(1-\hat s)^{-\alpha}}_{{}\le 2^{-\alpha}}\,d\hat s
\le \int_{-2}^{-1} (\Pi^m\chi^{m-2})'(\hat s)\,d\hat s =1.
$$
For $m>3$, we also need to estimate $\F''_m$, which involves
$\Pi^m\chi^{m-2}(\hat s)=\frac12(\hat s+2)(\hat s+3)$ on $(-3,-2)$,
and is bounded similarly to
$\A_m''$ in \eqref{AB_calculation_A2},
which yields
 $0\le \F_m''\le \A_m''$. Hence, we get the final assertion $\F_m\le 1+\A''_m$.
\end{proof}

\begin{corollary}[Uniform temporal mesh]\label{cor_uni_suff}
Let $\tau_j=\tau=TM^{-1}$ $\forall j\ge1$ and, using the notation \eqref{AB_def_uniform}, set $\beta_j:=\beta:=\frac{\theta}2 \nu\A'/\B$ $\forall\,j\ge 1$
with any $\theta\in[\frac12
,1]$.
Then
$\beta\in(0,\frac23)$, and the operator $\delta_t^\alpha$ enjoys the inverse-monotone representation
\eqref{UV_ab}.
\end{corollary}

\begin{proof}
By \eqref{AB_def_uniform}, one has $\beta= {\theta}\nu\frac{2\alpha}{\alpha+2}\in(0,\frac 23)$
$\forall\,\alpha\in(0,1)$, $\forall\,\theta\in(0,1]$.

By Corollary~\ref{cor_suf_cond}, for \eqref{UV_ab} it suffices to check
conditions \eqref{key_AB}.
For $m\ge 1$ condition \eqref{key_AB_1} is straightforward in view of
$\B_1=\A_1>0$ from \eqref{AB_bounds_uniform}.
For $m\ge 2$, \eqref{AB_bounds_uniform} yields
$\B_m=\B$ and $\A_m-\B_m+\F_m\le \A'-\B+1$, while $\A_m\ge \nu \A'$
implies
$\A_m-\beta_m\B_m=\A_m-\frac{\theta}2 \nu\A'\ge (1-\frac{\theta}2)\, \nu\A'> 0$. So \eqref{key_AB_1} follows,
while for \eqref{key_AB_2} it suffices to show that
$$
\textstyle
 (\A'-\B+1) - \beta\,(1-\frac{\theta}2)\, \nu\A' <0.
$$
Recall that $\beta=\frac{\theta}2 \nu\A'/\B$, so multiplying the above inequality by $4\nu^{-2}\B/\A'^2$, one gets
\beq\label{strict_eta0}
\theta(2-\theta)> \underbrace{4\,(\B/\A')\,\bigl(1-\B/\A'+1/\A'\bigr)}_{=\frac14(\alpha+2)\mbox{~~by~\eqref{AB_def_uniform}}}{}\cdot\nu^{-2}.
\eeq
The latter, and hence \eqref{key_AB_2}, is satisfied if
$$\textstyle
|\theta-1|
< \sqrt{1-\frac14(\alpha+2)\,\nu^{-2}}
\;\;\;\Leftarrow\;\;\;
\theta\in (\theta_0(\alpha),1],
\quad\theta_0(\alpha):=1-\sqrt{1-\frac14(\alpha+2)\,\nu^{-2}}.
$$
Here
$\theta_0(\alpha)<\frac12$
follows from $\alpha+2<3\,\nu^{2}$  $\forall\,\alpha\in(0,1)$.
\end{proof}

\subsection{General temporal meshes}
Now we shall
estimate the quantities in \eqref{AB_def} and check the inverse-monotonicity conditions \eqref{key_AB}
for more general  meshes.

\begin{lemma}[General temporal mesh]\label{lem_gen_AB}
Suppose that 
{\color{blue}$\sigma_j\ge\sigma_{j+1}\ge0$ $\forall\,j\ge2$.}
Then for the quantities in \eqref{AB_def} one has
$\B_1=\A_1>0$ and
$\forall\,m\ge 2$
\beq\label{AB_bounds_graded}
\B_m=\B-\frac{\sigma_m}{2(1+\sigma_m)}\A'>{
\frac{2\B} 3},\quad\; \A_m=\frac{\A'}{1-\sigma_m^2}-\A''_m\ge \nu \frac{\A'}{1-\sigma_m^2},\quad\; \F_m\le 1+\A''_m,
\eeq
where we use the notation
\eqref{AB_def_uniform}  and
$\sigma_m\in[0,1)$  from \eqref{grid_notation}.
\end{lemma}

\begin{proof}
We shall imitate the proof of Lemma~\ref{lem_uniform_AB} making  appropriate changes for $m\ge2$.
Rewrite all integrals  in terms of the  variable $\hat s:=(s-\frac12[t_{m-2}+t_m])/\ttau_m$, so  the interval $(t_{m-2},t_m)$ is mapped to $(-1,1)$,
but $s=t_{m-1}$ is now mapped to $\hat s=-\sigma_m$.

The evaluation of $\B_m$ is similar to \eqref{AB_calculation_B}, but now
(to ensure $\Pi^{m}\phi^m=0$ at $\hat s=-\sigma_m$)
one has $\Pi^{m}\phi^m(s)=\frac12\hat s(\hat s+1)+\frac12(1-\hat s^2)\sigma_m/(1+\sigma_m)$ on $(-1,1)$,
which yields the desired assertion for $\B_m$.

Next, similarly to \eqref{AB_calculation_A},  split $\A_m=\A'_m-\A''_m$,
where now $\Pi^{m}\phi^{m-1}(\hat s)=(1-\hat s^2)/(1-\sigma_m^2)$ on $(-1,1)$
(so that $\Pi^{m}\phi^{m-1}=1$ at $\hat s=-\sigma_m$), so we get a version of \eqref{AB_calculation_A}
with $\A'$ replaced by $\A'_m=\A'/(1-\sigma_m^2)$.
As to $\A_m''$ for $m>2$, it is estimated exactly as in \eqref{AB_calculation_A2}, only now the support of $\Pi^m \phi^{m-1}(\hat s)$ for $\hat s<-1$
is limited to a certain subset 
of $(\sigma_m-2,-1)$ (in view of $\tau_j\le \tau_{j+1}$ $\forall j\ge1$),
so $\int_{-2}^{-1}|\Pi^m \phi^{m-1}(\hat s)|d\hat s\le\frac1{12}$, which leads to the same upper bound for $\A''_m$ as in Lemma~\ref{lem_uniform_AB}.

The estimation of $\F_m$ remains as  the proof of Lemma~\ref{lem_uniform_AB};
in particular, we again enjoy $\F_m''\le \A_m''$ in view of 
{\color{blue}$\sigma_j\ge\sigma_{j+1}\ge0$ $\forall\,j\ge2$ (as the latter implies $\rho_j\ge\rho_{j+1}\ge1$).}

Finally,
 $\B_m>\frac23\B$ for  $m\ge 2$  follows from $\frac{\sigma_m}{2(1+\sigma_m)}\le\frac14$ $\forall\,\sigma_m\in[0,1)$
combined with the definitions of $\B$ and $\A'$ in \eqref{AB_def_uniform}.
\end{proof}

\begin{corollary}[General temporal mesh]\label{cor_inv_monot_general}
Let the temporal mesh satisfy $\sigma_j\ge\sigma_{j+1}\ge0$
$\forall\,j\ge2$, and for any $\theta\in[\frac12,1]$ set
\beq\label{def_eta}\textstyle
\eta(\sigma):=
(1-\sigma^2)\bigl[\B/\A'-\frac{\sigma}{2(1+\sigma)}\bigr],
\quad\;\;
\beta_1:=\beta_2,\quad\;\;
\beta_j:=\frac{\theta}2\nu/\eta(\sigma_j)\quad\forall\,j\ge2,
\eeq
where we use the notation
\eqref{AB_def_uniform}  and
$\sigma_j\in[0,1)$  from \eqref{grid_notation}.
Then $\beta_j\ge\beta_{j+1}>0$ $\forall\,j\ge1$.
Furthermore,
for any $\theta\in[\frac12,1]$
 there exists  $\bar \sigma=\bar \sigma(\alpha,\theta)\in(0,1)$ such that if $\sigma_j\in[0,\bar \sigma]$ $\forall j\ge 2$, then
$\beta_j\in(0,1)$  $\forall j\ge 1$
 and the operator $\delta_t^\alpha$ enjoys the inverse-monotone representation
\eqref{UV_ab}.
\end{corollary}

\begin{proof}
Note that $\eta(\sigma)>0$ $\forall\,\sigma\in[0,1)$, in view of $\B_m>0$ $\forall\,\sigma_m\in[0,1)$ in \eqref{AB_bounds_graded}. Hence $\beta_j>0$ $\forall j\ge1$.
Also $\eta$
is a decreasing function of $\sigma$, so $\sigma_j\ge\sigma_{j+1}\ge0$
$\forall\,j\ge2$ implies $\beta_j\ge\beta_{j+1}>0$ $\forall\,j\ge1$.

Next, note that, by \eqref{grid_notation}, $\sigma_j\ge\sigma_{j+1}\ge0$ implies $\rho_j\ge\rho_{j+1}\ge1$ $\forall\,j\ge2$.
So, by Corollary~\ref{cor_suf_cond}, for \eqref{UV_ab} it suffices to check
conditions \eqref{key_AB}.
For $m\ge 1$ condition \eqref{key_AB_1} is straightforward in view of
$\B_1=\A_1>0$ (provided that $\beta_1=\beta_2<1$, which will be shown below).
For $m\ge 2$, \eqref{AB_bounds_graded} yields
$\B_m=\eta(\sigma_m)\frac{\A'}{1-\sigma_m^2}$, so
$\beta_m\B_m=\frac{\theta}2\nu \frac{\A'}{1-\sigma_m^2}$,  while $\A_m\ge \nu \frac{\A'}{1-\sigma_m^2}$
implies
$\A_m-\beta_m\B_m\ge (1-\frac{\theta}2)\, \nu\frac{\A'}{1-\sigma_m^2}> 0$,
so \eqref{key_AB_1} follows.
For \eqref{key_AB_2}
also using
$\A_m-\B_m+\F_m\le [1-\eta(\sigma_m)]\frac{\A'}{1-\sigma_m^2}+1$,
we conclude that it suffices to show that
$$\textstyle
  [1-\eta(\sigma_m)]\,\frac{\A'}{1-\sigma_m^2}+1 -  \underbrace{\beta_{m-1}}_{\ge \beta_m}\,(1-\frac{\theta}2)\, \nu\,\frac{\A'}{1-\sigma_m^2}\le 0.
$$
Dividing this by $\frac{\A'}{1-\sigma_m^2}$ and multiplying by $4\eta(\sigma_m)\,\nu^{-2}$, and also using $\beta_m=\frac{\theta}2\nu/\eta(\sigma_m)$,
we find that
\eqref{key_AB_2} is satisfied if
\beq\label{nonstrict_eta}
\theta(2-\theta)\ge 4\eta(\sigma_m) \Bigl(1-\eta(\sigma_m)+(1-\sigma_m^2)/\A'\Bigr)\cdot\nu^{-2}.
\eeq
Comparing this to \eqref{strict_eta0} and also noting that $\eta(0)=\B/\A'$, we see that if $\sigma_m=0$, then  a strict version of \eqref{nonstrict_eta}
becomes \eqref{strict_eta0},
so,
as was shown in the proof of Corollary~\ref{cor_uni_suff}, 
it is satisfied $\forall\,\theta\in[\frac12,1]$.
Also, if $\sigma_m=0$, then $\beta_m=\beta<\frac23$ (where $\beta$ is defined in Corollary~\ref{cor_uni_suff}).
Consequently, $\forall\,\theta\in[\frac12,1]$ there exists $\bar\sigma(\alpha,\theta)\in(0,1)$ such that both \eqref{nonstrict_eta} and $\beta_m<1$ are satisfied $\forall\,m\ge2$ if $\sigma_m\in[0,\bar\sigma]$
$\forall\,m\ge2$.
(The computation of $\bar\sigma(\alpha,\theta)$ is discussed in Remark~\ref{rem_bar_sigma} below.)
\end{proof}

\begin{remark}\label{rem_kappa_ratio}
Under the conditions of Corollary~\ref{cor_inv_monot_general},  $\forall\,m\ge 3$, one has
\begin{align*}
\beta_m\frac{\kappa_{m,m}}{1-\beta_m}\cdot\frac{1-\beta_{m-1}}{|\kappa_{m,m-1}|}
&=\frac{\beta_m\B_m}{\A_m-\beta_m\B_m}\le \frac{\theta}{2-\theta}\,,
\\[0.2cm]
\frac{\kappa_{m,m}}{1-\beta_m}\cdot\frac{1-\beta_{m-1}}{\kappa_{m-1,m-1}}
&=\frac{\ttau_{m}^{-\alpha}\,\B_m}{\ttau_{m-1}^{-\alpha}\,\B_{m-1}}\ge \frac{\ttau_{m-1}^\alpha}{\ttau_{m}^{\alpha}}\,,
\end{align*}
where we used \eqref{kappa_obersve} and the observations on $\beta_m\B_m$ and $\A_m-\beta_m\B_m$ made in the proof of Corollary~\ref{cor_inv_monot_general}.
For the second relation, we also relied on $\{\B_m\}_{m=2}^M$ being a decreasing function of $\sigma_m$, in view of \eqref{AB_bounds_graded}.%
\end{remark}%

\begin{remark}[Computation of $\bar\sigma$]\label{rem_bar_sigma}
Using the notation $\eta_m=\eta(\sigma_m)$, one can rewrite \eqref{nonstrict_eta}
as
\beq\label{bar_sigma_rel2}
4\eta_m \bigl(1+a-\eta_m\bigr)\le b,
\quad\mbox{where}\quad
a:=(1-\sigma_m^2)/\A'>0,\quad b:= \nu^2\theta(2-\theta)<1,
\eeq
which is equivalent to
\beq\label{bar_sigma_rel}
2\eta_m\ge \textstyle
(1+a)+\sqrt{(1+a)^2-b}>1.
\eeq
Importantly, this also ensures that $\beta_m<(2\eta_m)^{-1}<1$.
Note that the remaining solutions of the quadratic inequality in \eqref{bar_sigma_rel2} are described by
$$
2\eta_m\!\le
(1+a)-\sqrt{(1+a)^2-b}=
\frac{b}{(1+a)+\!\sqrt{(1+a)^2-b}}
<\frac{\theta\nu(2-\theta)}{1+\!\sqrt{1-\theta(2-\theta)}}=\theta\nu,
$$
which corresponds to $\theta\nu\beta_m^{-1}=2\eta_m<\theta\nu$ or $\beta_m>1$, so such solutions are of no interest.
Going back to \eqref{bar_sigma_rel}, in which we use the definitions of $\eta(\sigma)$ from \eqref{def_eta}
and $a$ from \eqref{bar_sigma_rel2},
we arrive at
$$
\textstyle
(1-\sigma_m^2)\bigl[(2\B-1)/\A'-\frac{\sigma_m}{(1+\sigma_m)}\bigr]\ge
1+\sqrt{(1+a)^2-b}.
$$
Consequently, we impose $\sigma_m\in[0,\bar\sigma] $, where $\bar\sigma\in(0,1)$
is the minimal solution of the equation (in which $b$ is from \eqref{bar_sigma_rel2})
$$
\textstyle
\underbrace{(1-\bar \sigma)\bigl[c(1+\bar\sigma)-{\bar\sigma}\bigr]}_{=:g_L(\bar\sigma)}=
\underbrace{1+\sqrt{(1+(1-\bar\sigma^2)/\A')^2-b}}_{=:g_R(\bar\sigma)}
\,,\quad
c:=\frac{2\B-1}{\A'}=\frac{2+5\alpha-\alpha^2}{4\alpha}>\frac32.
$$
{\color{blue}Recall that setting $\sigma=0$ yields a strict inequality $g_L(0)>g_R(0)$.}
Also note that
$g_L(\sigma)$ is a parabola with zeros at $1$ and $-1-\frac{1}{c-1}$, so it is decreasing for positive $\sigma$, while
$g_R(\sigma)$ is also decreasing, and $\color{blue}g_R(1)>0$. So for each fixed $\alpha$ and $\theta$,
 starting with $\bar\sigma^{[0]}:=0$, the  iterative procedure $g_L(\bar\sigma^{[q+1]})=g_R(\bar\sigma^{[q]})$
 will generate an increasing sequence $\bar\sigma^{[q]}\in(0,1)$ converging to $\bar\sigma$.
 Finally, note that $\theta=1$ will produce the least restrictive $\bar\sigma$
 {\color{blue}(as then $b$ takes its maximal value)}.
\end{remark}

\section{Stability properties for the discrete fractional-derivative operator}\label{sec_stab}
In this section we shall combine the inverse-monotonicity of the operator $\delta_t^\alpha$
established in \S\ref{sec_inv_monot} with the barrier-function stability analysis developed in \cite{NK_XM}
for quasi-graded temporal meshes.

\begin{theorem}[Discrete comparison principle]\label{theo_DCPr}
Let the temporal mesh satisfy $\sigma_j\ge\sigma_{j+1}\ge0$
$\forall\,j\ge2$.
There exists  $\bar \sigma=\bar \sigma(\alpha)\in(0,1)$ such that if, additionally, $\sigma_j\in[0,\bar \sigma]$ $\forall j\ge 2$,
then the following statements are true.

(i)
 If $U^0\ge 0$ and $\delta_t^\alpha U^m\ge 0$  $\forall\, m\ge1$, then $U^j\ge 0$ for $\forall\, j\ge0$.

(ii)
If
for a certain barrier function $\{B^j\}_{j=0}^M$ one has
$|U^0|\le B^0$ and $|\delta_t^\alpha U^m|\le \delta_t^\alpha B^m$ $\forall\, m\ge1$, then  $|U^j|\le B^j$  $\forall\, j\ge0$.

(iii)
If $U^0=0$, then
$\displaystyle|U^m|\lesssim \max_{j=1,\ldots,m}\bigl\{t_j^{\alpha}\, |\delta_t^\alpha U^j|\bigr\}$
$\forall\,m\ge 1$.
\end{theorem}

\begin{proof}
Let $\bar\sigma$ be from Corollary~\ref{cor_inv_monot_general}
(for any $\theta\in[\frac12,1]$, e.g., $\theta=1$).
Then
the operator $\delta_t^\alpha$ enjoys the inverse-monotone representation
\eqref{UV_ab}, which will play the crucial role in our proof.
\smallskip

(i)
For $\{V^j\}$ from \eqref{UV_ab}, one has
$V^0=U^0\ge 0$, so $\delta_t^\alpha U^m\ge 0$ $\forall\, m\ge1$ implies $V^j\ge 0$  $\forall\, j\ge0$,
from which we then conclude that $U^j\ge 0$ for $\forall\, j\ge0$.
(Alternatively, the proof may directly employ the inverse monotonicity of the matrix associated with $\delta_t^\alpha$; see Remark~\ref{rem_M_matrices}.)
\smallskip

(ii) As the operator $\delta_t^\alpha$ is linear, the result follows from part (i).
\smallskip

(iii) For $\{V^j\}$ from \eqref{UV_ab},
we claim that $|V^m|\lesssim \max_{j=1,\ldots,m}\bigl\{t_j^{\alpha}\, |\delta_t^\alpha U^j|\bigr\}$.
To show this, note that $V^0=U^0=0$, so
$|V^1|= \kappa_{1,1}^{-1}|\delta_t^\alpha U^1|$
and $|V^2|\le |V^1|+\kappa_{2,2}^{-1}|\delta_t^\alpha U^2|$,
where, by \eqref{kappa_obersve},\,\eqref{AB_bounds_graded}, $\kappa_{1,1}\simeq t_1^{-\alpha}$ and $\kappa_{2,2}\simeq \ttau_2^{-\alpha}\simeq
t_2^{-\alpha}$, so for $m=1,2$ the desired bound on $|V^m|$ follows.
If $|V^n|=\max_{j\le m} |V^j|$ for some $3\le n\le m$, then
$\sum_{j=1}^n\kappa_{n,j} |V^n|\le |\delta_t^\alpha U^n|$,
where $\sum_{j=1}^n\kappa_{n,j}=-\kappa_{n,0}\gtrsim t_n^{-\alpha}$, in view of \eqref{kappa0},
so again
$ |V^m|\le|V^n|\lesssim t_n^{\alpha}|\delta_t^\alpha U^n|\le \max_{j\le m}\bigl\{t_j^{\alpha}\, |\delta_t^\alpha U^j|\bigr\}$.

Next, a similar argument shows that if
$\max_{j\le m}|U^j|= |U^k|$ for some $k\le m$, then  $ |U^k|\le |V^k|$.
Consequently, $|U^m|\le|U^k|\lesssim \max_{j=1,\ldots,k}\bigl\{t_j^{\alpha}\, |\delta_t^\alpha U^j|\bigr\}$.
\end{proof}

\begin{theorem}[Quasi-graded temporal grid]\label{theo_main_stab_XM}
Given $\gamma\in\R$,
let  the temporal mesh 
satisfy
\beq\label{t_grid_gen_XM}
\tau _1\simeq M^{-r},\qquad 
\tau_j
\simeq t_j/j,
\qquad 
t_j \sim \tau_1 j^r
\qquad
\forall\,j=1,\ldots,M
\eeq
for some $1\le r\le (3-\alpha)/\alpha$
if $\gamma>\alpha-1$ or for some $r\ge1$ if $\gamma\le \alpha-1$.
Additionally, 
let the temporal mesh satisfy $\sigma_j\ge\sigma_{j+1}\ge0$
$\forall\,j\ge2$
and
$\sigma_j\in[0,\bar \sigma]$ $\forall j\ge K+1$,
where
  $\bar \sigma\in(0,1)$ is from
Theorem~\ref{theo_DCPr}, and $1\le K\lesssim 1$
(i.e. $K$ is sufficiently large, but independent of $M$).
Then 
for $\{U^j\}_{j=0}^M$
 one has
\beq\label{main_stab_XM}
\left.\!\!\!\!\begin{array}{c}
|\delta_t^\alpha U^j|\lesssim (\tau_1/ t_j)^{\gamma+1}
\\[0.2cm]
\forall j\ge1,\;\;\; U^0=0
\end{array}\hspace{-0.15cm}\right\}
 \Rightarrow
|U^j|\lesssim
\U^j(\tau_1;\gamma)
:=
\tau_1 t_j^{\alpha-1}\!\left\{\!\!\begin{array}{ll}
1&\!\!\!\mbox{if~}\gamma>0\\
1+\ln(t_j/\tau_1)&\!\!\!\mbox{if~}\gamma=0\\
(\tau_1/t_j)^\gamma
&\!\!\!\mbox{if~}\gamma<0
\end{array}\right.
\hspace{-0.3cm}
%
%
\eeq
~\hfill $\forall j\ge 1$.
\end{theorem}

\begin{proof}
(i) First, consider the case $K=1$.
If  $1\le r\le (3-\alpha)/\alpha$,
note that mesh assumptions \eqref{t_grid_gen_XM} are equivalent to those in \cite[(2.1)]{NK_XM},
so the desired assertion is obtained by an application of Theorem~\ref{theo_DCPr}(ii)
with the barrier function $\{B^j\}$ from \cite[proofs of Theorems~2.1(i) and~4.2(i)]{NK_XM}.
If  $\gamma\le \alpha-1$, then \eqref{main_stab_XM} can be shown
(without assuming \eqref{t_grid_gen_XM})
by an application of Theorem~\ref{theo_DCPr}(iii)
imitating the proof of \cite[Theorem~2.1(ii)]{NK_XM}.

(ii)
Next, consider the case $K>1$.
As $K\lesssim 1$, by \eqref{t_grid_gen_XM}, one has $\tau_j\simeq \tau_1$ $\forall j\le K$.
So for $m\le K$ a calculation yields
$|U^m|\lesssim \sum_{j=0}^{m-1}|U^j|+\tau_1^\alpha|\delta_t^\alpha U^m|$
(in particular, $\kappa_{m,m}\simeq \tau_1^{-\alpha}$ follows from \eqref{kappa_obersve},\,\eqref{AB_bounds_graded}).
As
$|\delta_t^\alpha U^m|\lesssim 1$, so one gets
$|U^m|\lesssim \tau_1^\alpha\simeq \U^m$ $\forall\,m\le K$.

It remains to estimate the values of $\{\mathring{U}^j\}_{j=0}^M:=\{0,\ldots,0,U^{K+1},\ldots,U^M\}$
(i.e. $\mathring{U}^j$ is set to $0$ for $j\le K$ and to $U^j$ otherwise).
Note that $\delta_t^\alpha \mathring{U}^m=0$ for $m\le K$ and $|\delta_t^\alpha \mathring{U}^m|\lesssim 1$ for $m= K+1,K+2$.
Consider $m> K+2$. 
By \eqref{delta_t_def}, one has
$\delta_t^\alpha \mathring{U}^m=\delta_t^\alpha {U}^m-D^\alpha_t \Pi^m[U-\mathring{U}](t_m)$.
As $\Pi^m[U-\mathring{U}]$ has support on $(0, t_{K+1})$, vanishes at $0$ and $ t_{K+1}$, while its absolute value $\lesssim \tau_1^\alpha$,
so, recalling \eqref{CaputoEquiv} and applying an integration by parts yields
$|D^\alpha_t \Pi^m[U-\mathring{U}](t_m)|\lesssim \tau_1^\alpha\int_0^{ t_{K+1}}(t_m-s)^{-\alpha-1}ds\lesssim (\tau_1/t_m)^{\alpha+1}$ (where we also used $ t_{K+1}\simeq \tau_1$).
Consequently, for $m\ge K+1$ one concludes that $|\delta_t^\alpha \mathring{U}^m|$ is $\lesssim (\tau_1/t_m)^{\gamma+1}$ if $\gamma\le\alpha$
and $\lesssim (\tau_1/t_m)^{\alpha+1}$ otherwise.

Finally, let $\mathring{\delta}_t^\alpha$ be the operator of type $\delta_t^\alpha$, but associated with the mesh $\{t_j\}_{j=K-1}^M$,
i.e.
for any $\{W^j\}_{j=K-1}^M$, set
$\mathring{\delta}_t^\alpha W^K:=\int_{t_{K-1}}^{t_K}(\Pi_{1,K}W)(t_K-s)^{-\alpha}ds$
and $\mathring{\delta}_t^\alpha W^m:=\int_{t_{K-1}}^{t_m}(\Pi^mW)(t_m-s)^{-\alpha}ds$ for $m>K$.
Then $\mathring{\delta}_t^\alpha \mathring{U}^K=0$, while
$|\mathring{\delta}_t^\alpha \mathring{U}^m|=|{\delta}_t^\alpha \mathring{U}^m|$ for $m>K$.
Importantly, the bound of type \eqref{main_stab_XM}, which we already proved for ${\delta}_t^\alpha$ for the case $K=1$, applies to $\mathring{\delta}_t^\alpha$.
In the latter bound, $j\ge K$ and $t_j$ is replaced by $t_j-t_{K-1}\simeq t_j$.
In particular, we conclude that if $\gamma\le\alpha$, then $|\mathring{U}^j|\lesssim \U^j(\tau_1,\gamma)$, while if
$\gamma>\alpha$, then $|\mathring{U}^j|\lesssim \U^j(\tau_1,\alpha)=\U^j(\tau_1,\gamma)$.
Combining our findings, one gets $|U^j|=|\mathring{U}^j|\lesssim \U^j(\tau_1,\gamma)$ $\forall\,j\ge K+1$, and hence \eqref{main_stab_XM} $\forall\,j\ge 1$.
\end{proof}

\begin{corollary}[Graded temporal grid]\label{cor_main_stab_XM}
Given $\gamma\in\R$,
let  the temporal mesh be defined by
$\{t_j=T(j/M)^r\}_{j=0}^M$ for some $1\le r\le (3-\alpha)/\alpha$
if $\gamma>\alpha-1$ or for some $r\ge1$ if $\gamma\le \alpha-1$.
Then  the conditions of Theorem~\ref{theo_main_stab_XM} on the mesh are satisfied, and so
\eqref{main_stab_XM} holds true 
for any $\{U^j\}_{j=0}^M$ with
 $U^0=0$.
\end{corollary}

\begin{proof}
Clearly, the mesh satisfies \eqref{t_grid_gen_XM}, as well as
$\sigma_j\ge\sigma_{j+1}\ge0$
$\forall\,j\ge2$.
So it remains to find $K\lesssim 1$ such that
$\sigma_j\in[0,\bar \sigma]$ $\forall j\ge K+1$.
For the latter, in view of \eqref{grid_notation}, the sequence $\{\sigma_j\}$, as well as the related sequence $\{\rho_j\}$, is decreasing, so
it suffices to satisfy
\beq\label{K_def}
\rho_{K+1}=\frac{\tau_{K+1}}{\tau_{K}}=\frac{(K+1)^r-K^r}{K^r-(K-1)^r}
=\frac{(1+1/K)^r-1}{1-(1-1/K)^r}\le  \bar\rho:=\frac{2}{1-\bar\sigma}-1.
\eeq
As $\bar\sigma$ is independent of $M$, clearly, one can
 always choose such sufficiently large $K=K(r,\bar\sigma)$ independently of $M$.
\end{proof}

\begin{remark}[Modified graded mesh]\label{rem_K_mesh}
Although, as shown by Corollary~\ref{cor_main_stab_XM}, the result of Theorem~\ref{theo_main_stab_XM} applies to the standard graded mesh,
but it may still be desirable for the operator $\delta_t^\alpha$ to enjoy the inverse-monotonicity property
of type \eqref{UV_ab} $\forall\,j\ge1$ (rather than $\forall\,j\ge K+1$).
This can be easily ensured by a simple modification of the graded scheme as follows. Let
\beq\label{eq_K_mesh}
t_j:=T\, \hat t_j/\hat t_M,\qquad\mbox{where}\qquad\textstyle
\hat t_j:=\bigl(\frac{j+K'} M\bigr)^r-\bigl(\frac{K'} M\bigr)^r, \qquad K':=K-1,
\eeq
with $K$ from \eqref{K_def}.
To compute $K=K(r, \bar\sigma)$, note that $\bar\sigma$ can be computed, as described in Remark~\ref{rem_bar_sigma}.
%
Note also that if $K=1$,  one gets the standard graded mesh, while  $K>1$
implies that $\hat t_M=(1+K'/M)^r-(K'/M)^r\approx 1+r K'/M$.
Clearly, Corollary~\ref{cor_main_stab_XM} also applies to the modified graded mesh.
\end{remark}

\begin{remark}[Inverse-monotone modification of $\delta_t^\alpha$]\label{rem_modi_delta}
Consider the standard graded temporal mesh $\{t_j=T(j/M)^r\}_{j=0}^M$ for some $r\ge 1$.
As an alternative to modifying this mesh, as described in Remark~\ref{rem_K_mesh}, one can
ensure the inverse-monotonicity  \eqref{UV_ab} $\forall\,j\ge1$
by
tweaking the definition of $\delta_t^\alpha$ in \eqref{delta_t_def}
for $m\le K$ only as follows. Reset $\Pi^m:=\Pi_{1,j}$  on $(t_{j-1},t_j)$ $\forall j\le m\le K$
(i.e. the inverse-monotone L1 discretization is used for $m\le K$).
With this modification, also reset $\beta_j:=\beta_{K+1}$ $\forall\,j=1,\ldots,K$ in \eqref{def_eta}.
Then all results of this paper, that are valid for the graded mesh, also hold true for the modified discrete fractional-derivative
operator (as can be shown by only minor modifications in the relevant proofs).
%
\end{remark}

We finish this section with a more subtle version of Theorem~\ref{theo_main_stab_XM}, which will be useful when considering the fractional-derivative parabolic case in \S\ref{sec_parabolic}.

\renewcommand{\thetheorema}{\ref{theo_main_stab_XM}${}^*$}
\begin{theorema}\label{theo_main_stab_XM_star}
Let $\bar\sigma$ and the set $\{\beta_j\}_{j=1}^M$ be from Corollary~\ref{cor_inv_monot_general} (for any  $\theta\in[\frac12,1]$), and  $\{\kappa_{m,j}\}$ be
the unique set of the coefficients in the corresponding representation \eqref{UV} for the operator $\delta_t^\alpha$.
Also, given $\gamma\in\R$, let the temporal mesh satisfy
the conditions of Theorem~\ref{theo_main_stab_XM} with $K=1$.
Then for $\{U^j\}_{j=0}^M$ and $\{W^j\}_{j=0}^M$ with
 $U^0=W^0=0$ the following is true:
\beq\label{star_eq}
\left\{\begin{array}{cc}
\displaystyle\sum_{j=0}^m \kappa_{m,j}W^j\lesssim(\tau_1/ t_m)^{\gamma+1}&\forall\,m\ge 1\\[0.4cm]
\displaystyle
\frac{|U^j|-\beta_j |U^{j-1}|}{1-\beta_j}\lesssim W^j
&\forall\,j\ge 1
\end{array}\right.
\quad\Rightarrow\quad
|U^j|\lesssim
\U^j(\tau_1;\gamma),
\eeq
where $\U^j$ is defined in \eqref{main_stab_XM}.
\end{theorema}

\begin{proof}
Note that the choice of $\bar\sigma$ and $\{\beta_j\}_{j=1}^M$
in Corollary~\ref{cor_inv_monot_general} ensures that the corresponding representation \eqref{UV} for the operator $\delta_t^\alpha$
satisfies \eqref{UV_kappa}, i.e. $\delta_t^\alpha$ is associated with an inverse-monotone matrix;
see Remark~\ref{rem_M_matrices}.
Using the notation of this remark, the assumptions in \eqref{star_eq}
become $A_1\vec{W}\lesssim \vec{F}$ and $A_2\vec{|U|}\lesssim\vec{W}$, where $F^m:=(\tau_1/ t_m)^{\gamma+1}$.
As $A_1$ and $A_2$ are inverse-monotone, so $\vec{W}\lesssim A_1^{-1}\vec{F}$, and then $\vec{|U|}\le A_2^{-1}\vec{W}\lesssim A_2^{-1}A_1^{-1}\vec{F}$.
On the other hand, Theorem~\ref{theo_main_stab_XM}  implies that $0\le A_2^{-1}A_1^{-1}\vec{F}\lesssim \vec{\U}$, which yields the desired assertion.
\end{proof}

\section{Error estimation for a simplest example\! (without spatial derivatives)}\label{sec_paradigm}

Consider a fractional-derivative problem without spatial derivatives together with its  discretization of type \eqref{delta_t_def}:%
\begin{subequations}\label{simplest}%
\begin{align}
D_t^\alpha u(t)&=f(t)&&\hspace{-1.6cm}\mbox{for}\;\;t\in(0,T],&&\hspace{-0.8cm} u(0)=u_0,
\\ \delta_{t}^{\alpha} U^m&=f(t_m)&&\hspace{-1.6cm}\mbox{for}\;\;m=1,\ldots,M,&&\hspace{-0.8cm} U^0=u_0.
\end{align}
\end{subequations}
Throughout this subsection, with slight abuse of notation, $\pt_t $ will be used for $\frac{d}{dt}$.

The main result of this section is the following theorem, to the proof of which we shall devote the remainder of the section.

\begin{theorem}\label{theo_simplest}
Let the temporal mesh satisfy
 \eqref{t_grid_gen_XM}
for some $r\ge 1$, and also
$\sigma_j\ge\sigma_{j+1}\ge0$
$\forall\,j\ge2$
and
$\sigma_j\in[0,\bar \sigma]$ $\forall j\ge K+1$,
where
  $\bar \sigma\in(0,1)$ is from
Theorem~\ref{theo_DCPr}, and $1\le K\lesssim 1$.
Suppose that $u$ and $\{U^m\}$  satisfy \eqref{simplest}, and $|\partial_t^l u |\lesssim 1+t^{\alpha-l}$ for $l = 1,3$ and $t\in(0,T]$.
Then $\forall\,m=1,\ldots,M$ one has
\beq\label{error}
|u(t_m)-U^m|\lesssim {\mathcal E}^m:=
\left\{\begin{array}{ll}
M^{-r}\,t_m^{\alpha-1}&\mbox{if~}1\le r<3-\alpha,\\[0.3cm]
M^{\alpha-3}\,t_m^{\alpha-1}[1+\ln(t_m/t_1)]&\mbox{if~}r=3-\alpha,\\[0.2cm]
M^{\alpha-3}\,t_m^{\alpha-(3-\alpha)/r}
&\mbox{if~}r>3-\alpha
\end{array}\right.
\eeq
\end{theorem}

\begin{remark}[Convergence in positive time]\label{rem_positive_time}
Consider $t_m\gtrsim 1$. Then
${\mathcal E}^m\simeq M^{-r}$ for $r<3-\alpha$ and ${\mathcal E}^{m}\simeq M^{\alpha-3}$ for $r>3-\alpha$,
i.e. in the latter case the optimal convergence rate is attained.
For $r=3-\alpha$ one gets an almost optimal convergence rate as now ${\mathcal E}^{m}\simeq M^{\alpha-3}\ln M$.
\end{remark}

\begin{remark}[Global convergence]\label{rem_global_time}
Note that $
\max_{m\ge1}{\mathcal E}^m\simeq {\mathcal E}^1\simeq \tau_1^\alpha\simeq M^{-\alpha r}$ for $\alpha\le (3-\alpha)/r$,
while
$\max_{m\ge1}{\mathcal E}^m\simeq {\mathcal E}^M\simeq M^{\alpha-3}$ otherwise.
Consequently, Theorem~\ref{theo_simplest} yields the global error bound
$|u(t_m)-U^m|\lesssim M^{-\min \{\alpha r,3-\alpha\}} $.
This implies that
the optimal grading parameter for global accuracy is $r=(3-\alpha)/\alpha$.
\end{remark}

\begin{remark}
Theorem~\ref{theo_simplest} applies to the standard graded mesh
$\{t_j=T(j/M)^r\}_{j=0}^M$
for any $r\ge1$ (in view of Corollary~\ref{cor_main_stab_XM}), as well as to the modified graded mesh \eqref{eq_K_mesh}.
Furthermore, the proof of this theorem can be easily extended to the case of the  modified
discrete fractional-derivative operator described in Remark~\ref{rem_modi_delta}.
\end{remark}

To prove Theorem~\ref{theo_simplest}, we first get an auxiliary result.

\begin{lemma}[Truncation error]\label{lem_truncA}
For a sufficiently smooth function $u$, let $r^m:=\delta_{t}^{\alpha} u(t_m)-D_t^\alpha u(t_m)$ $\forall\,m\ge 1$, and
\begin{subequations}\label{psi_defA}
\begin{align}\label{psi_1_defA}
\psi^1&:=\sup_{s\in(0,t_2)}\!\!\bigl(s^{1-\alpha}|\pt_s u(s)|\bigr)+t_2^{-\alpha}{\rm osc}\bigl(u, [0,t_2]\bigr),
\\\label{psi_j_defA}
\psi^j&:= t_j^{3-\alpha}\!\! \sup_{s\in(t_{j-1},t_{j+1})}\!\!\!|\pt_s^3u(s)|\quad\qquad\forall\,2\le j\le M-1,
\end{align}
\end{subequations}
{\color{blue}where ${\rm osc}(u, [0,t_2]):=\sup_{[0,t_2]} u-\inf_{[0,t_2]} u$.}
Then, under conditions \eqref{t_grid_gen_XM} on the temporal mesh, one has
\beq\label{tr_er_boundA}
|r^m|\lesssim (\tau_1/t_m)^{\min\{\alpha+1,\,  (3-\alpha)/r\}} \max_{j\le \max\{1, m-1\}}\bigl\{ \psi^{j}\bigr\}
\qquad\forall\,m\ge 1.
\eeq
\end{lemma}\vspace{-0.3cm}

\begin{proof}
We closely imitate the proof of \cite[Lemma~4.7]{NK_XM}, so  some details will be skipped here.
From \eqref{delta_t_def}, recall that $\delta_{t}^{\alpha} u(t_m)=D^\alpha_t (\Pi^mu)(t_m)$.
Next, recalling the definition \eqref{CaputoEquiv} 
of $D^\alpha_t$, 
with the auxiliary function $\chi:=u-\Pi^mu$, we arrive at
$$
\Gamma(1-\alpha)\,r^m\!=\!
\int_0^{t_m}
\!\!\!\!(t_m-s)^{-\alpha}\underbrace{\pt_s[\Pi^mu(s)- u(s)]}_{{}=-\chi'(s)}\,ds
=\alpha
\!\int_0^{t_m}
\!\!\!\!(t_m-s)^{-\alpha-1}\chi(s)\,ds.
$$
Split the above integral to intervals $(0,t_1)$ and $(t_1,t_m)$.
On $(0,t_1)$ note that $\chi(t_1)=0$ implies $\chi(s)=-\int_s^{t_1}\chi'(\zeta)d\zeta$,
where
$|\chi'|\le|\pt_su|+|\pt_s (\Pi^mu)|$, while
$|\pt_s (\Pi^mu)|\lesssim t_2^{-1}{\rm osc}(u, [0,t_2])\le s^{\alpha-1}t_2^{-\alpha}{\rm osc}(u, [0,t_2])$
(in view of $\tau_1\simeq \tau_2$),
so a calculation yields
$|\chi(s)|\lesssim s^{\alpha-1}(t_1-s)\psi^1$.
Next, on any $(t_{j-1},t_j)$ for $1< j<m$ one has
$|\chi|\lesssim \tau_j^3 t_j^{\alpha-3}\psi^j$.
Finally, on $(t_{m-1},t_m)$, if $m>2$, then
$|\chi|\lesssim \tau^2_m(t_m-s)t_m^{\alpha-3}\psi^{m-1}$, while
if $m=2$, then we imitate the estimation on $(0,t_1)$ and again get
 $|\chi(s)|\lesssim s^{\alpha-1}(t_2-s)\psi^1
 \lesssim\tau^2_2(t_2-s)t_2^{\alpha-3}\psi^{1}$.

Combining our findings on $\chi$, a calculation shows that we get
the following version of \cite[(4.8)]{NK_XM}:
\beq\label{r_m_simplestA}
|r^m|\lesssim\mathring{\mathcal J}^m\,(\tau_1/t_m)^{\alpha+1}\,\psi^1+
{\mathcal J}^m\max_{j=2,\ldots,m}\bigl\{\nu_{m,j}(\tau_j/t_j)^{3-\alpha} (t_j/t_m)^{\alpha+1}
\,\psi^{j^*}\bigr\}
.
\eeq
Note that in various places here we also used $  t_{j-1}\simeq t_{j}\simeq s$ for $s\in(t_{j-1},t_{j})$, $j> 1$.
The notation in~\eqref{r_m_simplestA} is as follows:
\begin{align*}
\mathring{\mathcal J}^m&:= (t_m/\tau_1)^{\alpha+1}\int_0^{t_1}\!\!s^{\alpha-1}
(t_1-s)\,
(t_m-s)^{-\alpha-1}
ds
\lesssim 1
,
\\
{\mathcal J}^m&:=\tau_m^\alpha\, t_m^{\alpha/r+1}\int_{t_1}^{t_m}\!
s^{-\alpha/r-1}
\,(t_m-s)^{-\alpha-1}\,\min\{1,(t_m-s)/\tau_m\}\,ds \lesssim 1
,
\\[0.2cm]
\nu_{m,j}&:=(\tau_j/\tau_m)^\alpha\,(t_j/t_m)^{-\alpha(1-1/r)} \simeq 1, 
\\[0.4cm]
{j^*}&:=\min\{j,m-1\}.
\end{align*}
Here the bound on $\nu_{m,j}$ follows from $\tau_j/\tau_m\simeq (t_j/t_m)^{1-1/r}$ (in view of \eqref{t_grid_gen_XM}).
For the estimation of quantities of type $\mathring{\mathcal J}^m$ and ${\mathcal J}^m$, we refer the reader to \cite{NK_MC_L1}.
In particular,
for $\mathring{\mathcal J}^m$, we first use the observation that $(t_1-s)/(t_m-s)\le t_1/t_m$ for $s\in(0,t_1)$.
Then for  $\mathring{\mathcal J}^m$ and ${\mathcal J}^m$,
it is helpful to respectively use the substitutions $\hat s=s/t_1$ and $\hat s=s/t_m$,
while for ${\mathcal J}^m$ we also employ $ (t_1/t_m)^{-\alpha/r}\sim (\tau_m/t_m)^{-\alpha}$ (also in view of \eqref{t_grid_gen_XM}).

Combining the above observations  with \eqref{r_m_simplestA}  yields
$$
|r^m|\lesssim \max_{j\le \max\{1, m-1\}}\bigl\{\underbrace{(\tau_j/t_j)^{3-\alpha}}_{\simeq(\tau_1/t_j)^{(3-\alpha)/r}} (t_j/t_m)^{\alpha+1}
\,\psi^j\bigr\},
$$
where we also used $\tau_j/t_j\simeq(\tau_1/t_j)^{1/r}$ (in view of \eqref{t_grid_gen_XM}).
The desired bound \eqref{tr_er_boundA} follows as $\tau_1\le t_j\le t_m$.
\end{proof}

{\it Proof of Theorem~\ref{theo_simplest}.~}
Consider the error $e^m:=u(t_m)-U^m$, for which \eqref{simplest} implies
$e^0=0$ and $\delta_t^\alpha e^m=r^m$ $\forall\,m\ge1$, where
the truncation error $r^m$ is from Lemma~\ref{lem_truncA} and hence satisfies \eqref{tr_er_boundA}.
Furthermore, combining \eqref{psi_defA} with \eqref{t_grid_gen_XM} yields $\psi^1\lesssim 1$ (in view of $|{\rm osc}(u, [0,t_2])|\le\int_0^{t_2}|\pt_s u|\,ds\lesssim t_2^{\alpha}$)
and $\psi^j\lesssim 1$ for $j\ge 2$ (in view of $s\simeq t_j$ for $s\in(t_{j-1},t_{j+1})$ for this case).
Consequently, we arrive at
\beq\label{r_m_gamma}
|r^m|\lesssim (\tau/t_m)^{\gamma+1}\quad\forall\,m\ge 1,
\qquad\mbox{where}\;\;
\gamma+1:=\min\{\alpha+1,(3-\alpha)/r\}.
\eeq

Next we apply \eqref{main_stab_XM} from Theorem~\ref{theo_main_stab_XM} to bound $e^m=u(t_m)-U^m$.
Consider three cases.
\smallskip

Case $1\le r<3-\alpha$. Then both $(3-\alpha)/r> 1$ and $\alpha+1>1$,
so $\gamma>0$.
An application of \eqref{main_stab_XM} for this case yields
$|e^m|\lesssim \tau_1\, t_m^{\alpha-1}$, where $\tau_1\sim M^{-r}$.
\smallskip

Case $r=3-\alpha$. Then  $(3-\alpha)/r= 1$, while $\alpha+1>1$,
so $\gamma=0$.
An application of \eqref{main_stab_XM} yields
$|e^m|\lesssim \tau_1\, t_m^{\alpha-1}[1+\ln(t_m/t_1)]$,
where $\tau_1\sim M^{-r}=M^{\alpha-3}$.
\smallskip

Case $r>3-\alpha$.
Now
$(3-\alpha)/r< 1$, while $\alpha+1>1$,
so \mbox{$\gamma+1=(3-\alpha)/r<1$}.
Another application of \eqref{main_stab_XM}
(where, importantly,  unless $r\le (3-\alpha)/\alpha$, one has $\gamma\ge \alpha -1$)
 yields
$|e^m|\lesssim \tau_1\, t_m^{\alpha-1}(\tau_1/t_m)^{(3-\alpha)/r-1}\sim
\tau_1^{(3-\alpha)/r}t_m^{\alpha-(3-\alpha)/r}$,
where
$\tau_1^{(3-\alpha)/r}\sim M^{\alpha-3}$.
\endproof
%

\begin{remark}[General initial singularity]\label{rem_gen_sing}\color{blue}
Theorem~\ref{theo_simplest} can  be extended to the more general case
$|\partial_t^l u |\lesssim 1+t^{\alpha_0-l}$, where $0<\alpha_0\le \alpha$, as follows.
First, one needs to replace $\alpha$ by $\alpha_0$ in the definition \eqref{psi_defA} of $\{\psi^j\}$,
which will again lead to $\psi^j\lesssim 1$ $\forall\,j\ge1$.
With these changes, a new version of the bound \eqref{tr_er_boundA} (in Lemma~\ref{lem_truncA}) on the truncation error $r^m$
needs to be derived. This will lead
to the related bound \eqref{r_m_gamma} (in the proof of Theorem~\ref{theo_simplest}) with a new $\gamma=\gamma(\alpha,\alpha_0)$.
Once the latter is established, a straightforward application of Theorem~\ref{theo_main_stab_XM} will lead to a version of \eqref{error}
with a new  ${\mathcal E}^m={\mathcal E}^m(\alpha,\alpha_0)$.
Similarly, Theorems~\ref{theo_semi} and~\ref{theo_full} in Section~\ref{sec_parabolic} below can be generalized for
$\|\partial_t^l u (\cdot, t)\|_{L_2(\Omega)}\lesssim 1+t^{\alpha_0-l}$, and will also include the new ${\mathcal E}^m={\mathcal E}^m(\alpha,\alpha_0)$.
The full details for this more general case will be presented elsewhere.
\end{remark}

\section{Error analysis for the parabolic case}\label{sec_parabolic}
In this section, we shall generalize the analysis of \S\ref{sec_paradigm} to problems with
variable coefficients and
spatial derivatives. Both semidiscretizations in time and fully discrete methods will be addressed.

\subsection{Error analysis for semidiscretizations in time}\label{sec_semidiscr}
Consider the semidiscretization of our problem~\eqref{problem} in time using the discrete fractional-derivative operator
$\delta_t^\alpha$ from \eqref{delta_t_def}:
\beq\label{semediscr_method}
\delta_t^\alpha U^j +\LL U^j= f(\cdot,t_j)\;\;\mbox{in}\;\Omega,\quad U^j=0\;\;\mbox{on}\;\pt\Omega\quad\forall\,j=1,\ldots,M;\qquad U^0=u_0.
\eeq

\begin{lemma}[Stability for parabolic case]\label{lem_semi_stab}
Given $\gamma\in\R$,
let  the temporal mesh 
satisfy
\eqref{t_grid_gen_XM}
for some $1\le r\le (3-\alpha)/\alpha$
if $\gamma>\alpha-1$ or for some $r\ge1$ if $\gamma\le \alpha-1$.
There exists $\bar \sigma^*=\bar \sigma^*(\alpha)\in(0,1)$ such that if,
additionally, 
the temporal mesh satisfies $\sigma_j\ge\sigma_{j+1}\ge0$
$\forall\,j\ge2$
and
$\sigma_j\in[0,\bar \sigma^*]$ $\forall j\ge K+1$,
where
$1\le K\lesssim 1$
(i.e. $K$ is sufficiently large, but independent of $M$),
then for $\{U^j\}_{j=0}^M$ from \eqref{semediscr_method} one has
\beq\label{semi_stab}
 %
 \left.\begin{array}{c}
\|f(\cdot,t_j)\|_{L_2(\Omega)}
 \lesssim (\tau_1/ t_j)^{\gamma+1}
\\[0.2cm]
\forall j\ge1,\;\;\; U^0=0\;\;\mbox{in}\;\bar\Omega
\end{array}\right\}
 \quad\Rightarrow\quad
 \|U^j\|_{L_2(\Omega)}\lesssim
\U^j(\tau_1;\gamma)\;\;\forall\, j\ge 1,
\eeq
where $\U^j$ is defined in \eqref{main_stab_XM}.
\end{lemma}

\begin{proof} Fix any $\theta\in[\frac12,1)$ and
let $\bar\sigma=\bar\sigma(\alpha,\theta)$ 
be from Corollary~\ref{cor_inv_monot_general}.

(i)
First, we shall prove that
there exists   $\bar \sigma^*\in(0, \bar \sigma]$ such that
$\sigma_j\in[0,\bar \sigma^*]$ $\forall j\ge K+1$ implies
\beq\label{theta_choice}
\left(|\kappa_{m,m-1}|^{-1}\frac{\beta_m}{1-\beta_m}\right)^{\!2}\le
\frac{\kappa^{-1}_{m,m}}{1-\beta_m}\cdot\frac{\kappa^{-1}_{m-1,m-1}}{1-\beta_{m-1}}
\qquad\forall m\ge K+2,
\eeq
where $\beta_m$ are defined by \eqref{def_eta} for $m> K$ and are equal to $\beta_{K+1}$ for $m\le K$,
while
$\{\kappa_{m,j}\}$ is
the unique set of the coefficients in the corresponding representation \eqref{UV} for the operator $\delta_t^\alpha$.
To check this, rewrite
\eqref{theta_choice} as
$$
\left(\beta_m\frac{\kappa_{m,m}}{1-\beta_m}\cdot\frac{1-\beta_{m-1}}{|\kappa_{m,m-1}|}\right)^{\!2}\!\le
\frac{\kappa_{m,m}}{1-\beta_m}\cdot\frac{1-\beta_{m-1}}{\kappa_{m-1,m-1}}
\quad\Leftarrow\quad
\left(\frac{\theta}{2-\theta}\right)^{\!\!2/\alpha}\!\!\le \frac{\ttau_{m-1}}{\ttau_{m}},
$$
where the implication follows from Remark~\ref{rem_kappa_ratio}.
The sequence $\{\sigma_j\}$ is decreasing, and hence, in view of \eqref{grid_notation}, the related sequence $\{\rho_j\}$ is also decreasing, so it suffices to check that
$$
\frac{\ttau_{K+2}}{\ttau_{K+1}}=\frac{\rho_{K+1}\tau_K+\rho_{K+2}\tau_{K+1}}{\tau_K+\tau_{K+1}}\le\rho_{K+1}\le\bar\rho^*:= \left(\frac{2-\theta}{\theta}\right)^{\!\!2/\alpha}\in(1,3^{2/\alpha}].
$$
From this, $\bar\sigma^*:=\min\bigl\{\bar\sigma,\,{\color{blue}1-\frac{2}{1+\bar\rho^*}}\bigr\}>0$ will yield \eqref{theta_choice}.

(ii)
%
Next, suppose that $K=1$, i.e. $\sigma_j\in[0,\bar \sigma^*]$ $\forall j\ge 2$. Then \eqref{theta_choice} holds true $\forall\,m\ge3$, while, in view of  Corollary~\ref{cor_inv_monot_general},
$\bar\sigma^*\le\bar\sigma$
implies that the operator $\delta_t^\alpha$ enjoys the inverse-monotone representation
\eqref{UV_ab}.
Now, using  \eqref{UV}  and the notation  $V^m=\frac{1}{1-\beta_m}U^m-\frac{\beta_m}{1-\beta_m}U^{m-1}$ and $f^m:=f(\cdot,t_m)$, we can rewrite \eqref{semediscr_method} as
$$
\kappa_{m,m} V^m+\LL U^m=|\kappa_{m,m-1}|V^{m-1}+\sum_{j=1}^{m-2}|\kappa_{m,j}|V^{j}+ f^m.
$$
Consider the inner product of the above and $V^m$ using the notation
$$
w^m:=\sqrt{\displaystyle\|V^m\|_{L_2(\Omega)}^2+{\textstyle\frac{\kappa_{m,m}^{-1}}{1-\beta_m}}\langle\LL U^m, U^m\rangle}.
$$
Then
\begin{align*}
\kappa_{m,m} (w^m)^2=&\underbrace{|\kappa_{m,m-1}|\langle V^{m-1},V^m\rangle+\frac{\beta^m}{1-\beta_m}\langle\LL U^m, U^{m-1}\rangle}_{=:|\kappa_{m,m-1}|\,Q^m}
\\[-0.3cm]
&
\hspace{4.4cm}{}+\sum_{j=1}^{m-2}|\kappa_{m,j}|\underbrace{\langle V^{m},V^j\rangle}_{\le w^m w^j}
+\!\!\!\!\underbrace{\langle V^m ,f^m\rangle}_{\le w^m\|f^m\|_{L_2(\Omega)}}\!\!\!\!.
\end{align*}
Here
$Q^1=0$ in view of $U^0=V^0=0$ and
$Q^m\le w^m w^{m-1}$ $\forall\,m\ge 3$ in view of \eqref{theta_choice}.
For $m=2$ there is a sufficiently large constant $1\le \bar C\lesssim 1$ such that $Q^2\le \bar C w^2 w^{1}$.
(For example, using a version of Remark~\ref{rem_kappa_ratio} for $m=2$ and imitating the argument in part (i), one can choose $\bar C=(\ttau_{2}^{-\alpha}\,\B_2)/(\tau_{1}^{-\alpha}\,\B_{1})$;
see also \eqref{AB_def_a} and \eqref{t_grid_gen_XM}.)
Now dividing by $w^m$ and recalling that, by~\eqref{UV_kappa}, $\kappa_{m,j}\le 0$ $\forall\,j<m$,
we get
$$
\kappa_{1,1} w^1\le \|f^1\|_{L_2(\Omega)},\quad\!
\kappa_{2,2} w^2+\kappa_{2,1} (\bar C w^1)\le \|f^2\|_{L_2(\Omega)},\quad\!
\sum_{j=1}^m\kappa_{m,j}w^m\le\|f^m\|_{L_2(\Omega)}.
$$
Set $W^1:=\bar C w^1$ and $W^j:=w^j$ otherwise. Then, in view of  $\bar C\ge1$ and $\kappa_{m,1}\le 1$ $\forall m\ge3$, we arrive at
$\sum_{j=1}^m\kappa_{m,j}W^j\lesssim \|f^m\|_{L_2(\Omega)} \lesssim (\tau_1/ t_m)^{\gamma+1}$ $\forall\,m\ge 1$, while $W^0=0$.
Also, $\frac{1}{1-\beta_j}\|U^j\|_{L_2(\Omega)}-\frac{\beta_j}{1-\beta_j}\|U^{j-1}\|_{L_2(\Omega)}\le \|V^j\|_{L_2(\Omega)}\le w^j\le W^j$
$\forall\,j\ge 1$.
Thus we conclude that the assumptions in \eqref{star_eq} are satisfied with $|U^j|$ replaced by $\|U^j\|_{L_2(\Omega)}$.
So an application of Theorem~{\ref{theo_main_stab_XM}${}^*$} yields
the desired assertion $ \|U^j\|_{L_2(\Omega)}\lesssim
\U^j(\tau_1;\gamma)$ $\forall\,j\ge 1$.

(iii)
It remains to consider the case $K>1$, which will be reduced to the case
 $K=1$
by imitating part (ii) in the proof of Theorem~\ref{theo_main_stab_XM}.
In particular,
 for $m\le K$ we now get
$\|U^m\|_{L_2(\Omega)}+\tau_1^\alpha \langle\LL U^m, U^m\rangle \lesssim \sum_{j=0}^{m-1}\|U^j\|_{L_2(\Omega)}+\tau_1^\alpha\|f^m\|_{L_2(\Omega)}$.
Here
$\|f^m\|_{L_2(\Omega)}\lesssim 1$, so
$\|U^m\|_{L_2(\Omega)}
\lesssim \tau_1^\alpha\simeq \U^m$ $\forall\,m\le K$.
For $m> K$, we proceed exactly as in part (ii) in the proof of Theorem~\ref{theo_main_stab_XM} and employ $\{\mathring{U}^j\}$ and $\mathring{\delta}^\alpha_t$.
\end{proof}

\begin{theorem}\label{theo_semi}
Let the temporal mesh satisfy
 \eqref{t_grid_gen_XM}
for some $r\ge 1$, and also
$\sigma_j\ge\sigma_{j+1}\ge0$
$\forall\,j\ge2$
and
$\sigma_j\in[0,\bar \sigma^*]$ $\forall j\ge K+1$,
where
  $\bar \sigma^*\in(0,1)$ is from
Lemma~\ref{lem_semi_stab}, and $1\le K\lesssim 1$.
Suppose that $u$ from \eqref{problem} satisfies $\|\partial_t^l u (\cdot, t)\|_{L_2(\Omega)}\lesssim 1+t^{\alpha-l}$ for $l = 1,3$ and $t\in(0,T]$.
Then for  $\{U^m\}$ from \eqref{semediscr_method}, one has
$$
\|u(\cdot,t_m)-U^m\|_{L_2(\Omega)}\lesssim {\mathcal E}^m
\qquad\forall\,m=1,\ldots,M,
$$
where ${\mathcal E}^m$ is from \eqref{error}.
\end{theorem}

\begin{remark}
Theorem~\ref{theo_semi} applies to the standard graded mesh
$\{t_j=T(j/M)^r\}_{j=0}^M$
for any $r\ge1$ (in view of Corollary~\ref{cor_main_stab_XM}), as well as to the modified graded mesh \eqref{eq_K_mesh}.
Furthermore, the proof of this theorem can be easily extended to the case of the  modified
discrete fractional-derivative operator described in Remark~\ref{rem_modi_delta}.
\end{remark}

\begin{proof}
Consider the error $e^m:=u(\cdot,t_m)-U^m$, for which \eqref{problem} and \eqref{semediscr_method} imply
$\delta_t^\alpha e^m+\LL e^m=r^m$ $\forall\,m\ge1$ and
$e^0=0$, where
the truncation error, defined by $r^m:=\delta_t^\alpha u(\cdot,t_m)-D_t^\alpha u(\cdot,t_m)$, is estimated in Lemma~\ref{lem_truncA} and hence satisfies \eqref{tr_er_boundA}.
In the latter
$\psi^j=\psi^j(x)$ is defined by \eqref{psi_defA}, in which $u(\cdot)$ is understood as $ u(x,\cdot)$  when evaluating $\pt_su$, $\pt_s^2 u$, etc.
Furthermore, combining \eqref{psi_defA} with \eqref{t_grid_gen_XM} yields $\|\psi^1\|_{L_2(\Omega)}\lesssim 1$
(in view of $\|{\rm osc}(u(\cdot, t), [0,t_2])\|_{L_2(\Omega)}\le\int_0^{t_2}\|\pt_s u\|_{L_2(\Omega)}ds\lesssim t_2^{\alpha}$)
and $\|\psi^j\|_{L_2(\Omega)}\lesssim 1$ for $j\ge 2$ (in view of $s\simeq t_j$ for $s\in(t_{j-1},t_{j+1})$ for this case).
Consequently, we get a version of \eqref{r_m_gamma}:
$\|r^m\|_{L_2(\Omega)}\lesssim (\tau/t_m)^{\gamma+1}$ $\forall\,m\ge 1$,
where
$\gamma+1:=\min\{\alpha+1,(3-\alpha)/r\}$.
It remains to apply the estimate of type \eqref{semi_stab} from Lemma~\ref{lem_semi_stab} to $\{e^j\}$ considering the three cases for $r$
as in the proof of Theorem~\ref{theo_simplest}.
\end{proof}

\subsection{Error analysis for full discretizations}
In this section, we
discretize \eqref{problem}--\eqref{LL_def}, posed in a general bounded  Lipschitz domain  $\Omega\subset\R^d$,
 by applying
 a standard finite element spatial approximation to the temporal semidiscretization~\eqref{semediscr_method}.
 Let $S_h \subset H_0^1(\Omega)\cap C(\bar\Omega)$ be a Lagrange finite element space of fixed degree $\ell\ge 1$ 
 relative to a
 quasiuniform simplicial triangulation
 $\mathcal T$ of $\Omega$.
 (To simplify the presentation, it will be assumed that the triangulation covers $\Omega$ exactly.)
 Now,  $\forall\, m=1,\ldots,M$, let $u^m_h \in S_h$ satisfy
 \beq\label{FE_problem}
\begin{array}{l}
\langle \delta_t^\alpha u_h^m,v_h\rangle +A (u_h^m,v_h)= \langle f(\cdot,t_m),v_h\rangle\qquad\forall v_h\in S_h
\end{array}
\eeq
with  some $u_h^0\approx u_0$.
Here $\langle \cdot, \cdot\rangle$ is the  $L_2(\Omega)$ inner product, while
 $A$ is the standard symmetric bilinear form associated with the elliptic operator $\LL$ (i.e. $A(v,w)=\langle\LL v,w\rangle$ for smooth $v$ and $w$ in $H_0^1(\Omega)$).

\begin{lemma}[Stability for full discretizations]\label{lem_full_stab}
Under the conditions of Lemma~\ref{lem_semi_stab} on the temporal mesh,
 for $\{u_h^j\}_{j=0}^M$ from \eqref{FE_problem} one has
\beq\label{full_stab}
 \left.\begin{array}{c}
\|f(\cdot,t_j)\|_{L_2(\Omega)}
 \lesssim (\tau_1/ t_j)^{\gamma+1}
\\[0.2cm]
\forall j\ge1,\;\;\; u_h^0=0\;\;\mbox{in}\;\bar\Omega
\end{array}\right\}
 \quad\Rightarrow\quad
 \|u_h^j\|_{L_2(\Omega)}\lesssim
\U^j(\tau_1;\gamma),
\eeq
where $\U^j$ is defined in \eqref{main_stab_XM}.
\end{lemma}

\begin{proof}
We closely imitate the proof of Lemma~\ref{lem_semi_stab} replacing $\{U^j\}$ everywhere by $\{u_h^j\}$, and also
employing \eqref{FE_problem} with
$v_h:=V^m=\frac{1}{1-\beta_m}u_h^m-\frac{\beta_m}{1-\beta_m}u_h^{m-1}$ instead of \eqref{semediscr_method}.
\end{proof}

Our error analysis will invoke the Ritz projection $\RR_h u(t)\in S_h$ of $u(\cdot,t)$
associated with our discretization of the operator ${\LL}$ and
defined by $A (\RR_h u,v_h)=\langle{\LL} u,v_h\rangle$ $\forall v_h\in S_h$ and $t\in[0,T]$.
Assuming that the domain is such that
$\|v\|_{W^2_2(\Omega)}\lesssim \|\LL v\|_{L_2(\Omega)}$ whenever $\LL v\in L_2(\Omega)$, for the error of the Ritz projection $\rho(\cdot, t)=\RR_h u(t)-u(\cdot, t)$ one has
\beq\label{Ritz_er_L2}
\|\pt_t^l \rho(\cdot, t)\|_{L_2(\Omega)}\lesssim h\inf_{v_h\in S_h}\|\pt_t^l u(\cdot, t)-v_h\|_{W^1_2(\Omega)}
\qquad\mbox{for}\;\;
l=0,1,\;\;t\in(0,T].
\eeq
For $l=0$, see, e.g., \cite[Theorem~5.7.6]{BrenScott}.
A similar result for $l=1$ follows as $\pt_t\rho(\cdot, t)=\RR_h \dot u(t)-\dot u(\cdot, t)$, where $\dot u:=\pt_t u$.

\begin{theorem}\label{theo_full}
Let the temporal mesh satisfy
 \eqref{t_grid_gen_XM}
for some $r\ge 1$, and also
$\sigma_j\ge\sigma_{j+1}\ge0$
$\forall\,j\ge2$
and
$\sigma_j\in[0,\bar \sigma^*]$ $\forall j\ge K+1$,
where
  $\bar \sigma^*\in(0,1)$ is from
Lemma~\ref{lem_semi_stab}, and $1\le K\lesssim 1$.
Suppose that $u$ from \eqref{problem} satisfies $\|\partial_t^l u (\cdot, t)\|_{L_2(\Omega)}\lesssim 1+t^{\alpha-l}$ for $l = 1,3$ and $t\in(0,T]$.
Then for  $\{u_h^m\}$ from \eqref{FE_problem}, subject to $u_h^0=\RR_h u_0$, one has
\beq\label{error_full}
\|u(\cdot,t_m)-u_h^m\|_{L_2(\Omega)}\lesssim{} {\mathcal E}^m+\|\rho(\cdot, t_m)\|_{L_2(\Omega)}+t_m^\alpha\sup_{t\in(0,t_{m})}\!\!\bigl\{t^{1-\alpha} \|\pt_t\rho(\cdot, t)\|_{L_2(\Omega)}\bigr\}
\eeq
$\forall\,m\ge 1$, where $\rho(\cdot, t):=\RR_h u(t)-u(\cdot, t)$, and ${\mathcal E}^m$ is from \eqref{error}.
\end{theorem}

\begin{proof}
Let $e_h^m:=\RR_h u(t_m)-u_h^m\in S_h$.
Then $u(\cdot,t_m)-u_h^m=e_h^m-\rho(\cdot, t_m)$, so
it suffices to prove the desired bounds for $e_h^m$.
Note that $e_h^0=0$, while
 a standard calculation using \eqref{FE_problem} and \eqref{problem} yields
\begin{align}\label{prob_for_e}
\langle \delta_t^\alpha e_h^m,v_h\rangle +A (e_h^m,v_h)
&=
\langle \delta_t^\alpha \underbrace{\RR_h u}_{=\rho+u}(t_m),v_h\rangle +\underbrace{A (\RR_h u(t_m),v_h)}_{{}=\langle{\LL} u(\cdot, t_m),v_h\rangle}
-\langle f(\cdot,t_m),v_h\rangle\hspace{-0.2cm}
\\[0.2cm]\notag
&=\langle
\delta_t^\alpha \rho(\cdot, t_m)+r^m
,v_h\rangle\qquad\forall v_h\in S_h\;\;\forall\,m\ge1.
\end{align}
Here $r^m=\delta_t^\alpha u(\cdot,t_m)-D_t^\alpha u(\cdot,t_m)$ is from the proof of Theorem~\ref{theo_semi},
where it was shown that $\|r^m\|_{L_2(\Omega)}\lesssim (\tau/t_m)^{\gamma+1}$ $\forall\,m\ge 1$
with
$\gamma+1:=\min\{\alpha+1,(3-\alpha)/r\}$.

Suppose that $\delta_t^\alpha \rho(\cdot, t_m)=0$ $\forall\,m$ in \eqref{prob_for_e}.
Then an application of  the estimate of type \eqref{full_stab} from Lemma~\ref{lem_full_stab}  to $\{e_h^j\}$, with the three cases for $r$ considered separately
as in the proof of Theorem~\ref{theo_simplest}, yields $\| e_h^m\|_{L_2(\Omega)}\lesssim {\mathcal E}^m$.

Next, suppose that  $r^m=0$ $\forall\,m$ in \eqref{prob_for_e}, and  $\sup_{t\in(0,T)}\{t^{1-\alpha}\|\pt_t \rho(\cdot, t)\|_{L_2(\Omega)}\}= 1$.
Then, by \eqref{CaputoEquiv},
$\|D_t^\alpha \rho(\cdot, t_m)\|_{L_2(\Omega)}\lesssim 1$. For $r_\rho^m:=\delta_t^\alpha \rho(\cdot, t_m)-D_t^\alpha \rho(\cdot, t_m)$,
a version of the truncation error estimation in Lemma~\ref{lem_truncA} yields
$$
|r_\rho^m|\lesssim (\tau/t_m)^{\min\{\alpha+1,\,  (1-\alpha)/r\}} \max_{j=1,\ldots,m-1}\bigl\{ \psi_\rho^{j}\bigr\},
$$
where $\{\psi_\rho^{j}\}$ are defined by versions of
\eqref{psi_defA} with $u$ replaced by $\rho$, and $3$ in two places in \eqref{psi_j_defA} replaced by $1$. So we conclude that
$\|r_\rho^m\|_{L_2(\Omega)}\lesssim 1$, and hence $\|\delta_t^\alpha \rho(\cdot, t_m)\|_{L_2(\Omega)}\lesssim 1$ $\forall\,m\ge 1$.
Now an application of  the estimate of type \eqref{full_stab} from Lemma~\ref{lem_full_stab}  to $\{e_h^j\}$,
with $\gamma+1=0$,
yields $\| e_h^m\|_{L_2(\Omega)}\lesssim \U^m(\tau_1;-1)=t_m^\alpha$, where we also used the definition of $\U^m$ from \eqref{main_stab_XM}.

As \eqref{prob_for_e} is a linear problem for $\{e_h^m\}$, combining our findings yields \eqref{error_full}.
\end{proof}

Recalling the error bounds~\eqref{Ritz_er_L2} for the the Ritz projection, one immediately gets the following result.

\begin{corollary}\label{cor_L2}
 Under the conditions of Theorem~\ref{theo_semi},
let $\|\pt_t^l u(\cdot,t)\|_{W^{\ell+1}_2(\Omega)}\lesssim 1+t^{\alpha-l}$ for $l=0,1$
and $t\in(0,T]$.
Then there exists a unique solution  $\{u_h^m\}_{m=1}^M$ of \eqref{FE_problem} subject to $u_h^0=\RR_h u_0$,
and
\beq\label{err_bound_cor}
\|u(\cdot,t_m)-u_h^m\|_{L_2(\Omega)}\lesssim {}{\mathcal E}^m+
h^{\ell+1}\qquad\forall\,m\ge 1,
\eeq
where ${\mathcal E}^m$ is from \eqref{error}.
\end{corollary}

\begin{remark} The above
Theorem~\ref{theo_semi} and Corollary~\ref{cor_L2} apply to the standard graded mesh
$\{t_j=T(j/M)^r\}_{j=0}^M$
for any $r\ge1$ (in view of Corollary~\ref{cor_main_stab_XM}), as well as to the modified graded mesh \eqref{eq_K_mesh}.
Furthermore, the proofs can be easily extended to the case of the  modified
discrete fractional-derivative operator described in Remark~\ref{rem_modi_delta}.
\end{remark}

\begin{remark}
The assumptions on the derivatives of $u$ made in Corollary~\ref{cor_L2},
as well as in Theorems~\ref{theo_simplest} and~\ref{theo_semi},
are realistic {\color{blue}under certain compatibility conditions};
see examples in \cite[\S6]{NK_MC_L1}.
\color{blue}
More generally,  $u_0\in W^{2}_2(\Omega)\cap H_0^1(\Omega)$ implies
$\|\pt^l_t u(\cdot,t)\|_{W^{2}_2(\Omega)}\lesssim t^{-l}$
\cite[Theorem~2.1]{laz_review}. Hence, for the case $\ell=1$ one gets
$\|\pt_t (\Pi^m\rho)(\cdot,t_m)\|_{L_2(\Omega)}\lesssim h^{2}(\tau+t_m)^{-1}$,
which, by \eqref{delta_t_def},  yields
$\|\delta_t^\alpha \rho(\cdot, t_m)\|_{L_2(\Omega)}\lesssim h^{2}|\ln\tau|\, t_m^{-\alpha}$, so 
the error bound of type \eqref{err_bound_cor} will now include the term
$h^{2}|\ln\tau|$.
If $u_0$ is less regular, a more careful analysis (such as used in the proof of \cite[Theorem~3.2]{laz_review})
is required to deal with the contribution to the error induced by $\delta_t^\alpha \rho$.
\end{remark}

\section{Numerical results}\label{sec_Num}

  \begin{figure}[b!]
\begin{center}
\includegraphics[height=0.38\textwidth]{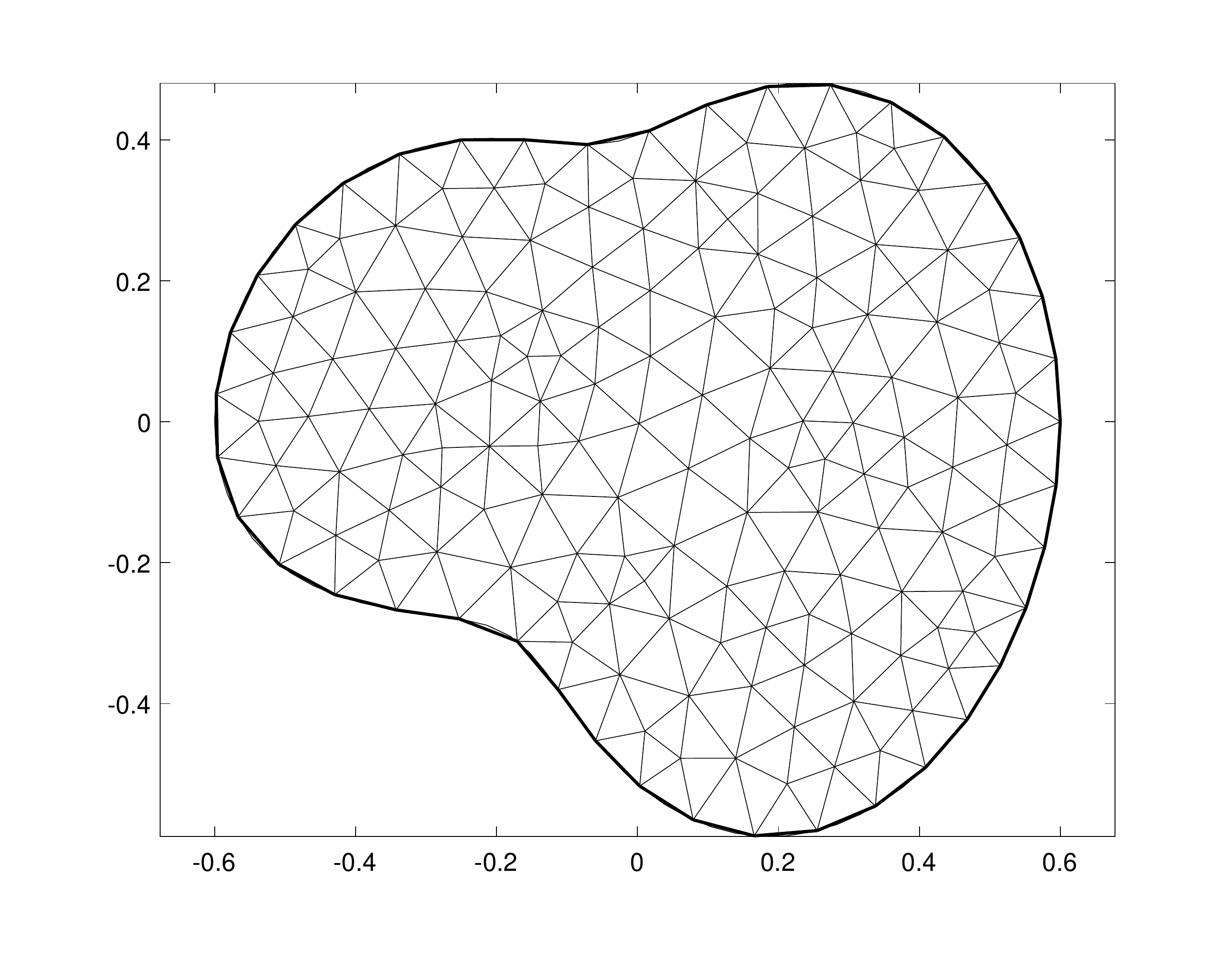}\hfill\includegraphics[height=0.38\textwidth]{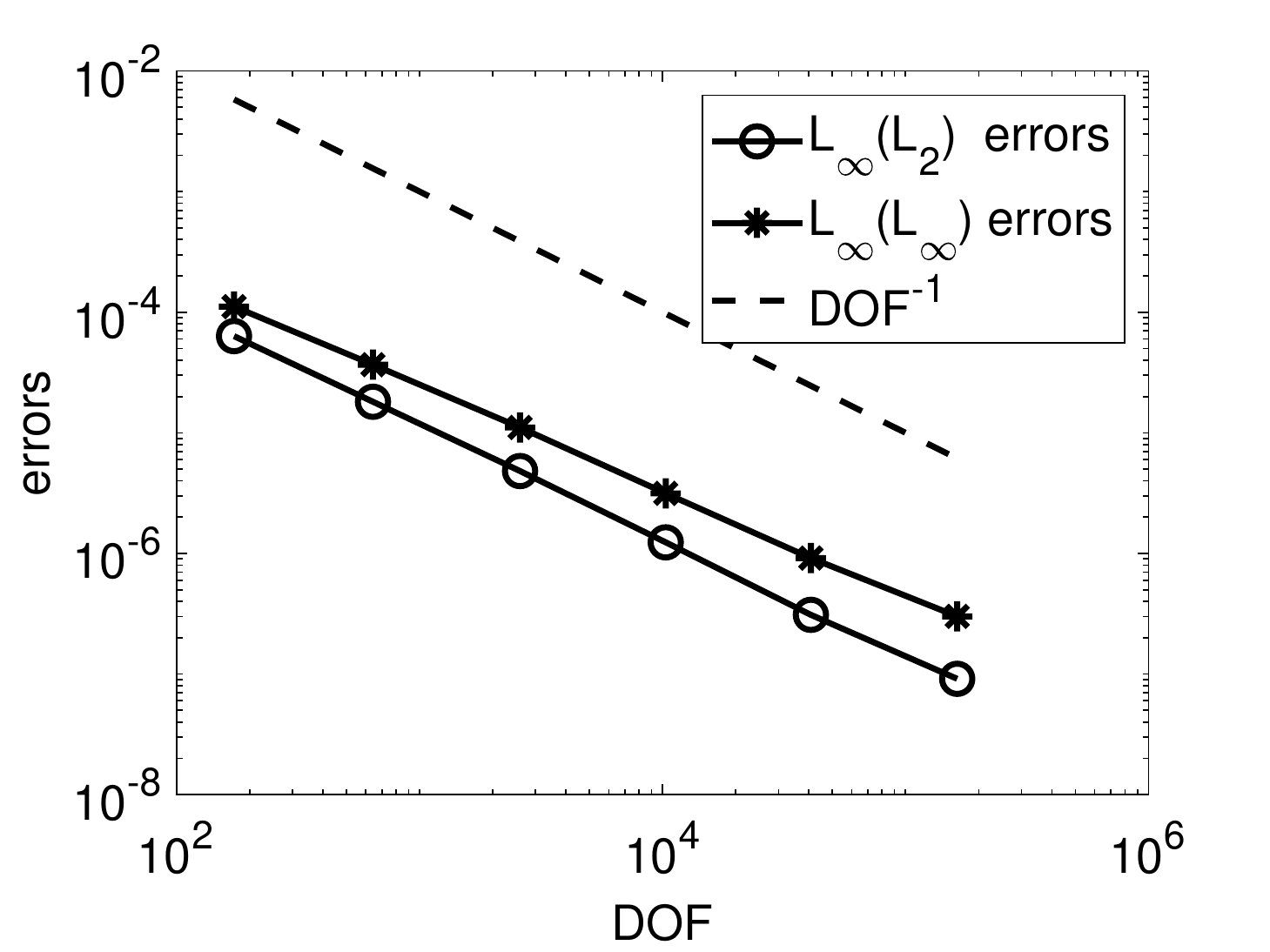}
\end{center}
\vspace{-0.3cm}
 \caption{\label{fig_mesh}\it\small
 Fractional-order parabolic test problem:
 Delaunay triangulation of $\Omega$ with DOF=172 (left),
maximum $L_2(\Omega)$ and $L_\infty(\Omega)$ errors for $\alpha=0.5$, $r=(3-\alpha)/\alpha$ and $M=2048$.}
 \end{figure}

\subsection{Parabolic case} Our fractional-order parabolic test problem  is \eqref{problem} with
$\LL=-(\pt_{x_1}^2+\pt_{x_2}^2)+\color{blue}(1+|x|^2)$, posed
in the domain $\Omega\times[0,1]$ (see Fig.\,\ref{fig_mesh}, left)
with $\pt\Omega$  parameterized by
$x_1(l):=\frac23R\cos\theta$ and $x_2(l):=R\sin\theta$, where
$R(l): =  0.4 + 0.5\cos^2\! l$ and
$\theta(l) := l + e^{(l-5)/2}\sin(l/2)\sin l$ for $l\in[0,2\pi]$; see \cite[\S7]{NK_MC_L1}.
We choose  $f$, as well as the initial and non-homogeneous boundary conditions, so that the unique exact solution
$u=t^{\alpha}\cos(xy)$.
This problem is discretized by \eqref{FE_problem} (with an obvious modification for the case of non-homogeneous boundary conditions)
using
lumped-mass linear finite elements 
on quasiuniform Delaunay triangulations of $\Omega$ (with DOF denoting the number of degrees of freedom in space).

The errors in the maximum $L_2(\Omega)$ norm
are shown in
Fig.\,\ref{fig_mesh} (right) and
Table~\ref{t2}
for, respectively, a large fixed $M$ and DOF.
In the latter case, we also give computational rates of convergence.
The errors were computed
using the piecewise-linear interpolant $u^I\in S_h$ in $\Omega$ of the exact solution
as
$\max_{m=1,\ldots, M}\|u_h-u^I\|_{L_2(\Omega)}$.
The graded temporal mesh
$\{t_j=T(j/M)^r\}_{j=0}^M$ was used
with the optimal $r= (3-\alpha)/\alpha$;  see Remark~\ref{rem_global_time}.
In view of the latter remark, by
Corollary~\ref{cor_L2},
the errors are expected to be $\lesssim M^{-(3-\alpha)}+h^2$, where $h^2\simeq \mbox{DOF}^{-1}$. 
Our numerical results clearly confirm the sharpness of this corollary for the considered case.

{\color{blue}Fig.\,\ref{fig_mesh} (right) also shows the errors in the maximum $L_\infty(\Omega)$ norm. Although it is not clear how
the error analysis of Section~\ref{sec_parabolic} can be generalized for this case,
it is worth noting that our numerical results (in Fig.\,\ref{fig_mesh}, as well as a version of Table~\ref{t2} for this case)
suggest that
the errors in the maximum $L_\infty(\Omega)$ norm are  $\lesssim M^{-(3-\alpha)}+h^2|\ln h|$.}

 {
\begin{table}[t!]
\begin{center}
\caption{Fractional-order parabolic test problem:
 maximum $L_2(\Omega)$ errors (odd rows) and
computational rates $q$ in $M^{-q}$ (even rows) for
 $r=(3-\alpha)/\alpha$ and spatial DOF=255435}
\label{t2}
\vspace{-0.2cm}
{\small\color{blue}
\begin{tabular}{rrrrrrrr}
\hline
\strut\rule{0pt}{9pt}&&
$M=32$&$M=64$& $M=128$&$M=256$&
$M=512$& $M=1024$\\
\hline
$\alpha=0.3$&
&4.885e-2	&8.787e-3	&1.426e-3	&2.239e-4	&3.470e-5	&5.353e-6	\\	
&&2.475	&2.623	&2.672	&2.690	&2.697	\\[3pt]	
$\alpha=0.5$&
&2.305e-3	&4.802e-4	&9.012e-5	&1.631e-5	&2.911e-6	&5.164e-7\\	
&&2.263	&2.414	&2.466	&2.486	&2.495\\[3pt]
$\alpha=0.7$&
%
&7.683e-4	&2.030e-4	&4.561e-5	&9.772e-6	&2.030e-6	&4.163e-7\\	
&&1.920	&2.154	&2.223	&2.267	&2.286\\	
\hline
\end{tabular}}
\end{center}
\end{table}
}

  \begin{figure}[b!]
\begin{center}
\includegraphics[height=0.30\textwidth]{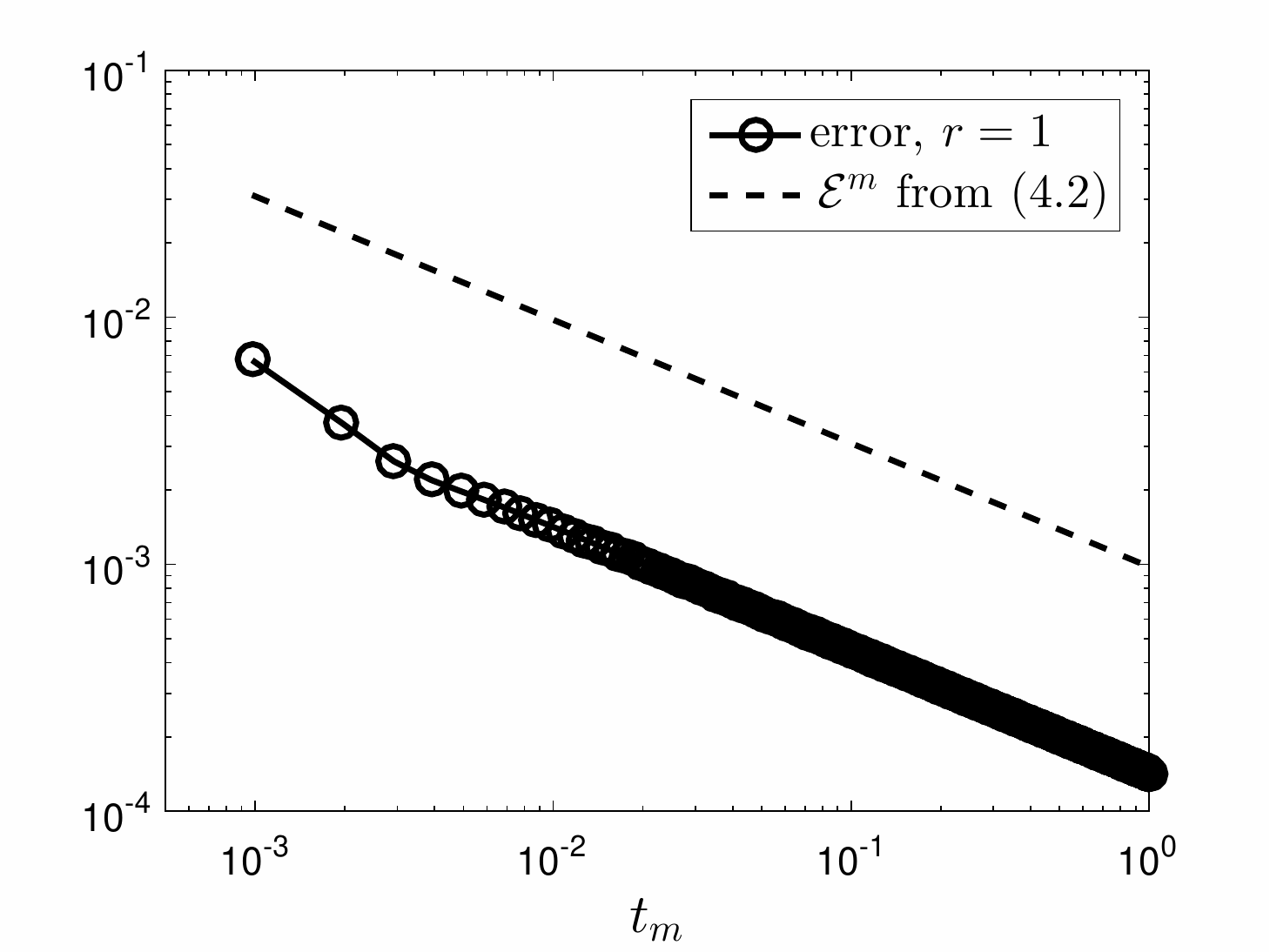}~~\includegraphics[height=0.30\textwidth]{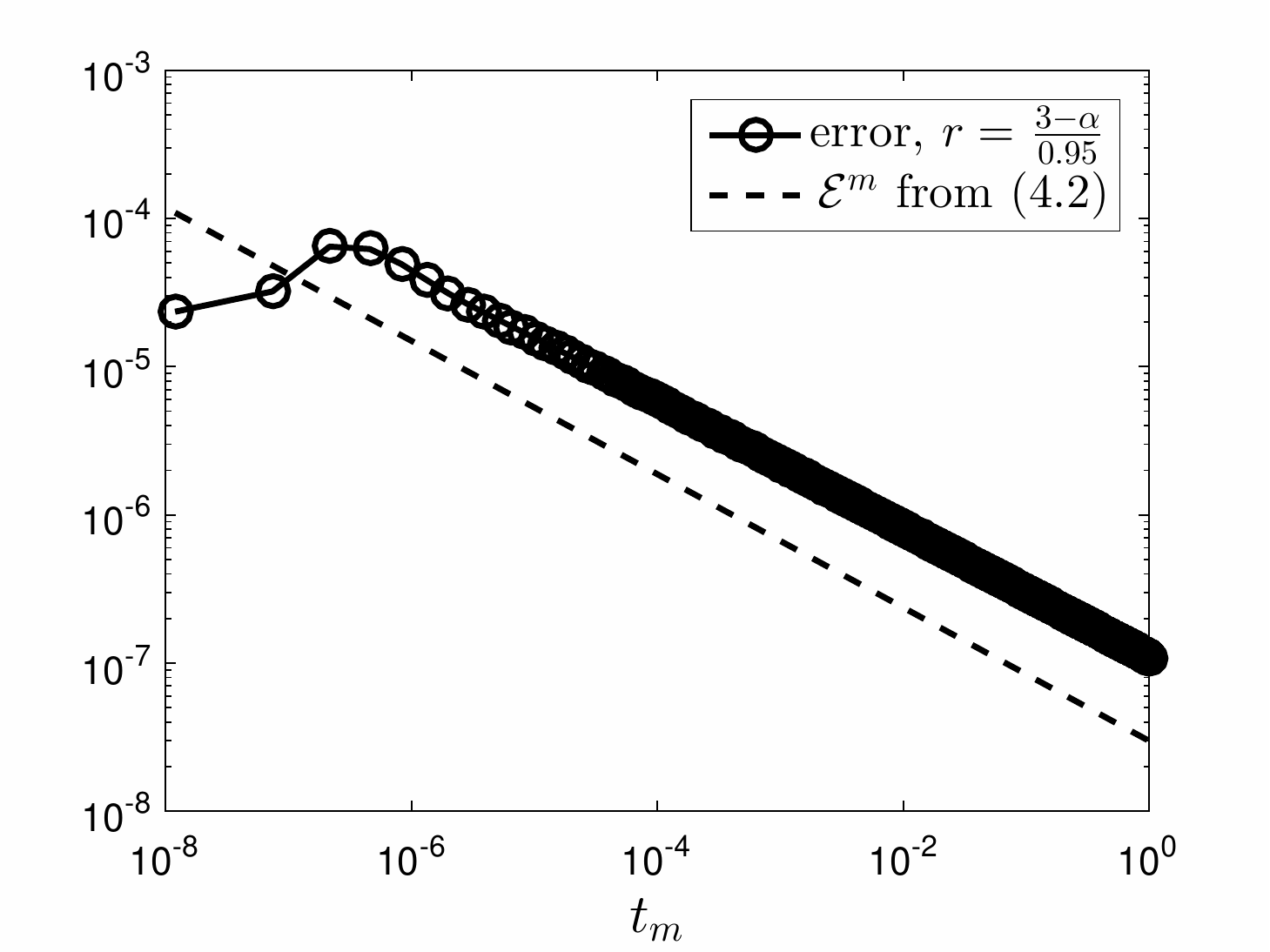}%
\\%
\includegraphics[height=0.30\textwidth]{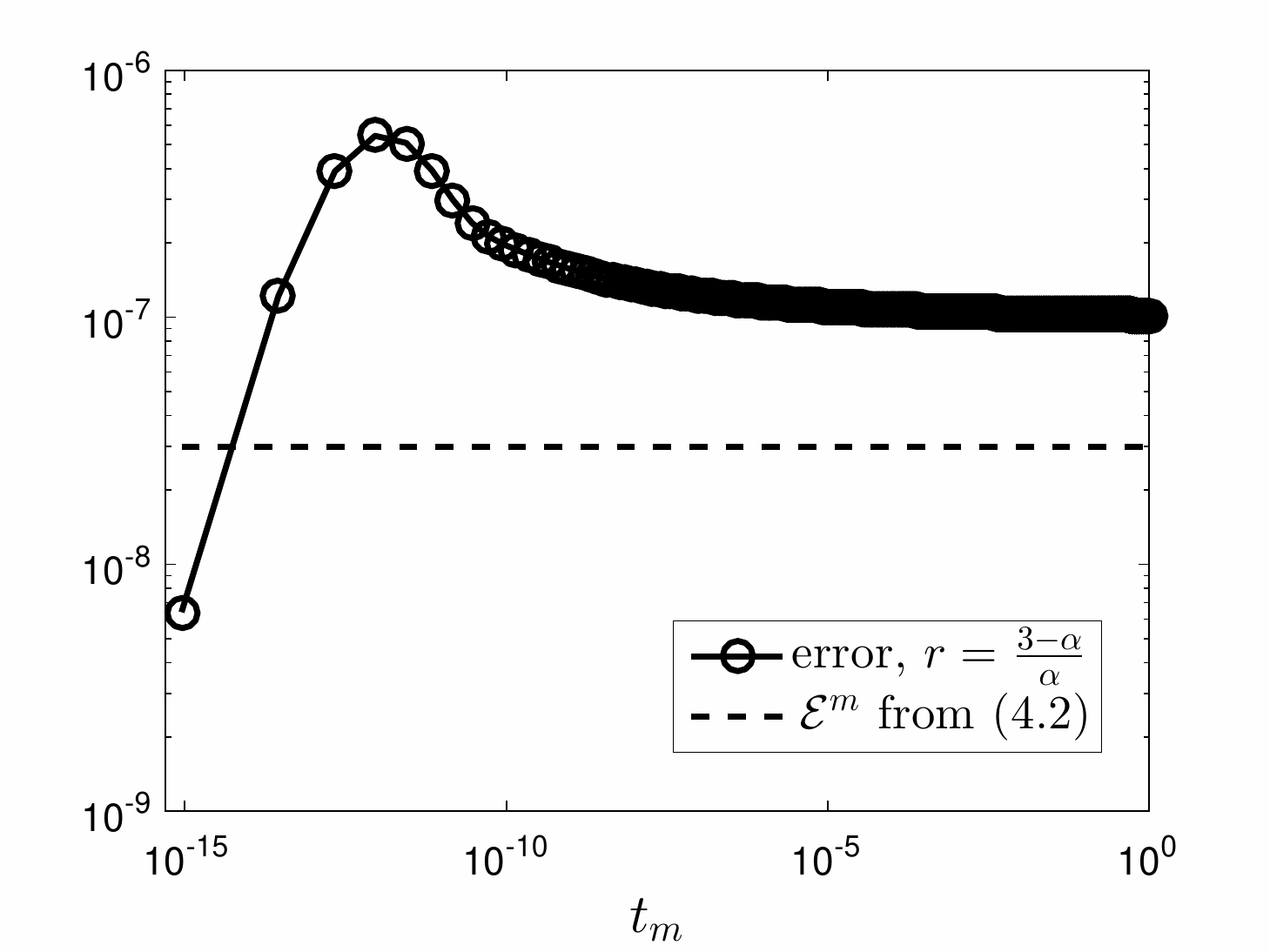}~~\includegraphics[height=0.30\textwidth]{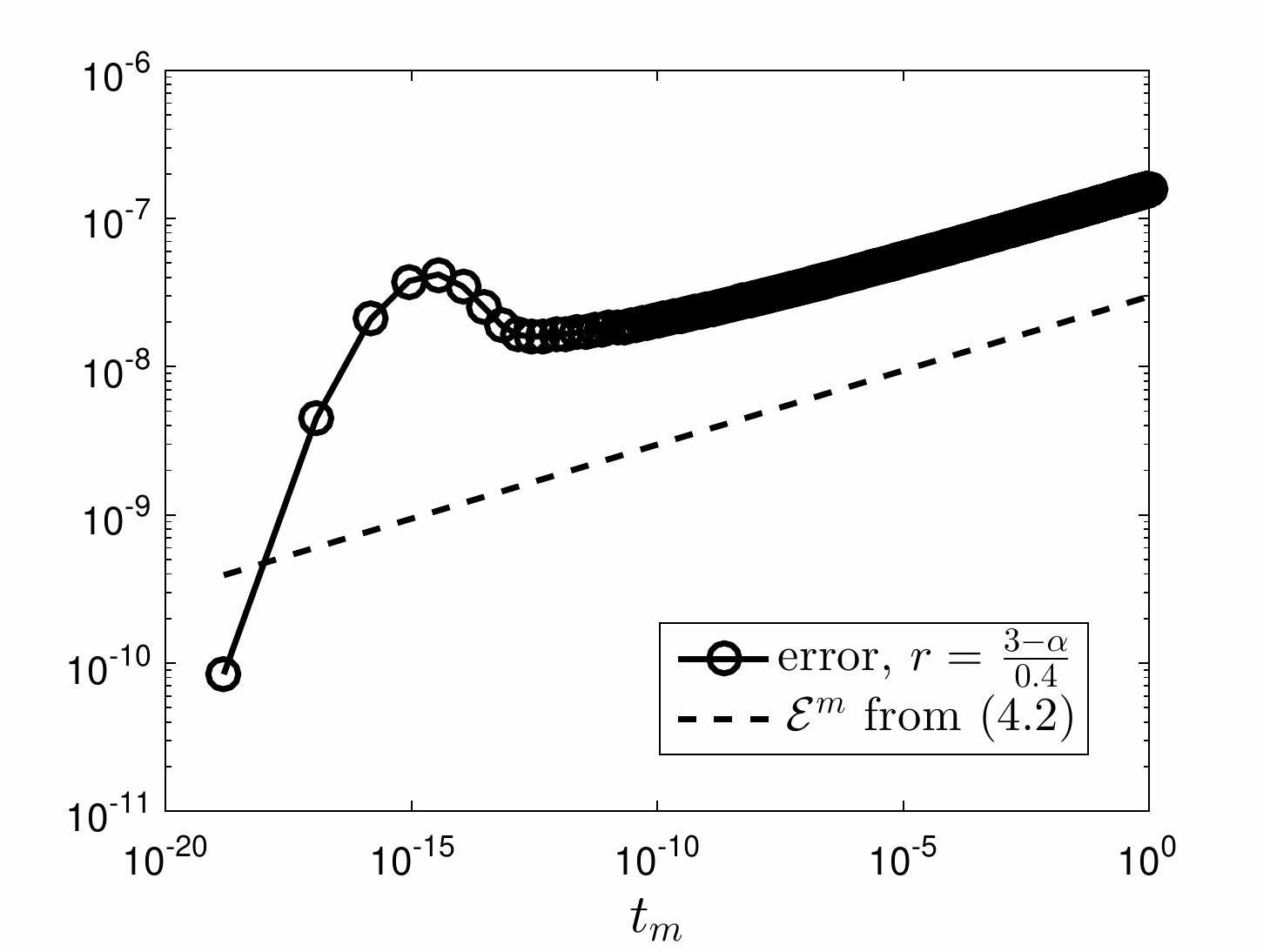}%
\end{center}
\vspace{-0.3cm}
 \caption{\label{pointwise_fig}\it\small
Initial-value test problem: pointwise errors for $\alpha=0.5$ and $M=1024$, cases $r=1$, $r=(3-\alpha)/0.95$, $r=(3-\alpha)/\alpha$ and $r=(3-\alpha)/0.4$.}
 \end{figure}

{
\begin{table}[h!]
\begin{center}
\caption{%
Initial-value test problem:
  errors at $t=1$ (odd rows) and
computational rates $q$ in $M^{-q}$ (even rows) for
 $r=1$, 
 $r=(3-\alpha)/.95$ and $r=(3-\alpha)/\alpha$}
\label{t0_positive_time}
\tabcolsep=4pt
\vspace{-0.2cm}
{\small
\begin{tabular}{lrrrrrrr}
\hline
\strut\rule{0pt}{9pt}&&
$\;\;\;M=2^5$&$M=2^7$& $M=2^9$&$M=2^{11}$&
 $M=2^{13}$&$M=2^{15}$\\
\hline
$r=1$
&$\alpha=0.3$
&3.324e-3	&8.297e-4	&2.073e-4	&5.182e-5	&1.296e-5	&3.239e-6\\
&&1.001	&1.000	&1.000	&1.000	&1.000\\[2pt]
&$\alpha=0.5$
&4.557e-3	&1.141e-3	&2.852e-4	&7.132e-5	&1.783e-5	&4.457e-6\\	
&&0.999	&1.000	&1.000	&1.000	&1.000\\[2pt]
%
%
&$\alpha=0.7$
&4.501e-3	&1.127e-3	&2.818e-4	&7.047e-5	&1.762e-5	&4.405e-6\\	
&&0.999	&1.000	&1.000	&1.000	&1.000\\		
%
\hline\strut\rule{0pt}{9pt}
$r=\frac{3-\alpha}{.95}$
&$\alpha=0.3$
&1.570e-4	&3.435e-6	&7.601e-8	&1.701e-9	&3.843e-11	&8.771e-13\\
&&2.757	&2.749	&2.741	&2.734	&2.727	\\[2pt]
&$\alpha=0.5$
&5.440e-4	&1.828e-5	&6.038e-7	&1.972e-8	&6.384e-10	&2.053e-11\\
&&2.447	&2.460	&2.468	&2.474	&2.480\\[2pt]	
%
%
&$\alpha=0.7$
&9.278e-4	&4.524e-5	&2.101e-6	&9.477e-8	&4.191e-9	&1.827e-10\\
&&2.179	&2.214	&2.235	&2.249	&2.260\\	
\hline
\strut\rule{0pt}{9pt}$r=\frac{3-\alpha}{\alpha}$
&$\alpha=0.3$
&8.360e-4	&1.481e-5	&2.950e-7	&6.248e-9	&1.373e-10	&3.088e-12\\	
&&2.910	&2.825	&2.781	&2.754	&2.737\\[2pt]	
&$\alpha=0.5$
&7.448e-4	&1.973e-5	&5.839e-7	&1.788e-8	&5.541e-10	&1.726e-11\\	
&&2.619	&2.539	&2.515	&2.506	&2.503\\[2pt]
%
%
&$\alpha=0.7$
&9.391e-4	&3.381e-5	&1.320e-6	&5.339e-8	&2.188e-9	&9.009e-11\\	
&&2.398	&2.340	&2.314	&2.304	&2.301\\
\hline
\end{tabular}}
\end{center}
\end{table}
}

{
\begin{table}[h!]
\begin{center}
\caption{%
Initial-value test problem:
 maximum nodal errors (odd rows) and
computational rates $q$ in $M^{-q}$ (even rows) for
 $r=1$, $r=3-\alpha$ and $r=(3-\alpha)/\alpha$}
\label{t0_global}
\tabcolsep=4pt
\vspace{-0.2cm}
{\small
\begin{tabular}{lrrrrrrr}
\hline
\strut\rule{0pt}{9pt}&&
$\;\;\;M=2^5$&$M=2^7$& $M=2^9$&$M=2^{11}$&
 $M=2^{13}$&$M=2^{15}$\\
\hline
$r=1$
&$\alpha=0.3$
&6.524e-2	&4.304e-2	&2.840e-2	&1.873e-2	&1.236e-2	&8.155e-3\\
&&0.300	&0.300	&0.300	&0.300	&0.300\\[2pt]
&$\alpha=0.5$
&3.794e-2	&1.897e-2	&9.484e-3	&4.742e-3	&2.371e-3	&1.186e-3\\	
&&0.500	&0.500	&0.500	&0.500	&0.500\\[2pt]	
%
%
&$\alpha=0.7$
&1.631e-2	&6.180e-3	&2.342e-3	&8.874e-4	&3.363e-4	&1.274e-4\\
&&0.700	&0.700	&0.700	&0.700	&0.700\\	
%
\hline\strut\rule{0pt}{9pt}
$r=3-\alpha\;\;$
&$\alpha=0.3$
&2.131e-2	&6.934e-3	&2.256e-3	&7.339e-4	&2.388e-4	&7.768e-5\\	
&&0.810	&0.810	&0.810	&0.810	&0.810\\[2pt]	
&$\alpha=0.5$
&6.185e-3	&1.093e-3	&1.933e-4	&3.417e-5	&6.040e-6	&1.068e-6\\	
&&1.250	&1.250	&1.250	&1.250	&1.250\\[2pt]	
%
%
&$\alpha=0.7$
&1.867e-3	&2.004e-4	&2.151e-5	&2.308e-6	&2.477e-7	&2.659e-8\\
&&1.610	&1.610	&1.610	&1.610	&1.610\\	
\hline
\strut\rule{0pt}{9pt}$r=\frac{3-\alpha}{\alpha}$
&$\alpha=0.3$
&6.510e-2	&1.542e-3	&3.652e-5	&8.648e-7	&2.048e-8	&4.851e-10\\	
&&2.700	&2.700	&2.700	&2.700	&2.700\\[2pt]	
&$\alpha=0.5$
&3.142e-3	&9.820e-5	&3.069e-6	&9.590e-8	&2.997e-9	&9.365e-11\\	
&&2.500	&2.500	&2.500	&2.500	&2.500\\[2pt]	
%
%
&$\alpha=0.7$
&1.273e-3	&5.247e-5	&2.164e-6	&8.922e-8	&3.679e-9	&1.517e-10\\	
&&2.300	&2.300	&2.300	&2.300	&2.300\\
\hline
\end{tabular}}
\end{center}
\end{table}
}

\subsection{Pointwise sharpness of error estimate for the initial-value problem}
Here, to demonstrate the sharpness of the error estimate (\ref{error}) given by Theorem~\ref{theo_simplest},
we consider the simplest initial-value fractional-derivative test problem \eqref{simplest} 
with the simplest typical exact solution $u(t):=t^\alpha$.
Table~\ref{t0_positive_time} shows the errors and the corresponding convergence rates at $t=1$, which agree with~(\ref{error}), in view of
Remark~\ref{rem_positive_time}.
In particular, the latter implies that the errors are
$\lesssim M^{-\min \{ r,3-\alpha\}}$ for $r\neq 3-\alpha$.
The maximum errors and corresponding convergence rates given in Table~\ref{t0_global}
clearly confirm the conclusions of Remark~\ref{rem_global_time},
which predicts from the pointwise  bound~(\ref{error}) that the global errors are
$\lesssim M^{-\min \{\alpha r,3-\alpha\}}$.
Furthermore, in Fig.~\ref{pointwise_fig}, the pointwise errors for various $r$ are compared with the pointwise theoretical error bound~(\ref{error}),
and again, with the exception of a few initial mesh nodes,  we observe remarkably good agreement.
Note that Fig.~\ref{pointwise_fig} only addresses the case $\alpha=0.5$, but for other values of $\alpha$ we observed similar consistency of (\ref{error})
with the actual pointwise errors.

\end{document}